\documentclass[oneside,english]{amsart}

\usepackage[T1]{fontenc}
\usepackage[latin9]{inputenc}
\usepackage{prettyref,hyperref}
\usepackage{amstext,amsthm,amssymb}
\usepackage{mathrsfs,bbm,esint}
\usepackage{babel}
\usepackage{graphicx}
\usepackage{fullpage}
\usepackage{cite}

\makeatletter
\numberwithin{equation}{section}
\numberwithin{figure}{section}
\theoremstyle{plain}
\newtheorem{thm}{\protect\theoremname}[section]
\theoremstyle{remark}
\newtheorem{rem}{\protect\remarkname}[section]
\theoremstyle{plain}
\newtheorem{prop}{\protect\propositionname}[section]
\ifx\proof\undefined\
  \newenvironment{proof}[1][\proofname]{\par
    \normalfont\topsep6\p@\@plus6\p@\relax
    \trivlist
    \itemindent\parindent
    \item[\hskip\labelsep
          \scshape
      #1]\ignorespaces
  }{%
    \endtrivlist\@endpefalse
  }
  \providecommand{\proofname}{Proof}
\fi
\theoremstyle{plain}
\newtheorem{lem}{\protect\lemmaname}[section]

\newcommand{\indicator}[1]{\mathbbm{1}_{#1}}

\newrefformat{prob}{Problem \ref{#1}}
\newrefformat{prop}{Proposition \ref{#1}}
\newrefformat{subsec}{Section \ref{#1}}
\newrefformat{rem}{Remark \ref{#1}}
\newrefformat{fig}{Figure \ref{#1}}

\makeatother

\providecommand{\lemmaname}{Lemma}
\providecommand{\propositionname}{Proposition}
\providecommand{\remarkname}{Remark}
\providecommand{\theoremname}{Theorem}

\begin{document}

\title{Optimal cooling of an internally heated disc}

\author{
Ian Tobasco}

\address{Department of Mathematics, Statistics, and Computer Science, University of Illinois at Chicago, Chicago, IL 60607, USA
}




\begin{abstract}
Motivated by the search for sharp bounds on turbulent heat transfer
as well as the design of optimal heat exchangers, we consider incompressible
flows that most efficiently cool an internally heated disc. Heat enters
via a distributed source, is passively advected and diffused, and
exits through the boundary at a fixed temperature. We seek an advecting
flow to optimize this exchange. Previous work on energy-constrained
cooling with a constant source has conjectured that global optimizers
should resemble convection rolls; we prove one-sided bounds on energy-constrained
cooling corresponding to, but not resolving, this conjecture. In the
case of an enstrophy constraint, our results are more complete: we
construct a family of self-similar, tree-like ``branching flows''
whose cooling we prove is within a logarithm of globally optimal. 
These results hold for general space- and time-dependent source-sink distributions that add more heat than they remove.
Our main technical tool is a non-local Dirichlet-like variational principle for bounding solutions of the
inhomogeneous advection-diffusion equation with a divergence-free velocity.
\end{abstract}



\dedicatory{To Charlie Doering, who taught us how to bound.}
\maketitle

\section{Introduction}

Imagine a fluid in a container that is heated from within, and whose
temperature is fixed at its boundary. How should the fluid flow so
that the container cools as quickly as possible? This question arises,
for instance, in the design of optimal heat exchangers, whose complicated
shapes and flows facilitate heat transfer at rates far beyond
diffusion \cite{alam2018comprehensive,alben2017improved,alben2017optimal,catton2011conjugate,feppon2019shape,feppon2020topology,feppon2021body,geb2013measuring,rastan2020heat,weisberg1992analysis}.
More generally, the problem is related to the ongoing search for sharp\emph{
}bounds on turbulent heat transfer in a variety of settings, including
internally heated \cite{arslan2021bounds,arslan2021bounds2,bouillaut2019,goluskin2015internally,goluskin2016internally,lepot2018radiative}
as well as buoyancy-driven convection \cite{drivas2021bounds,fantuzzi2018bounds,goluskin2016bounds,toppaladoddi2021thermal,whitehead2011ultimate},
amongst others. While any bound must explain why the rate of heat
transfer associated to a given flow cannot exceed a certain amount,
sharp bounds are distinguished in that their values are attained by
particular flows\textemdash flows that, as a result, are extremal.
Finding sharp bounds on convective heat transfer, or at least bounds
that are asymptotically sharp in their scaling with respect to physical parameters, has remained
a widely open problem in our understanding of turbulence since the
pioneering works of Howard and Busse \cite{busse1969howards,howard1963heat}
and Doering and Constantin \cite{doering1996variational} on Rayleigh-B\'enard
convection (for a recent summary of the state of the art, see \cite{doering2020turning}).

A widely used technique for proving bounds on fluid dynamical quantities
is the ``background method'' introduced by Doering and Constantin
\cite{constantin1995variational,doering1994variational,doering1996variational},
which takes the form of a convex variational problem whose optimal
value is the bound. Whether or not this method and its descendants
(see, e.g., \cite{kerswell1998unification,plasting2003improved} or \cite{fantuzzi2021background} for a recent review) produce
sharp bounds on convection remains largely unclear. This has led some to
ask whether information beyond the background method might be
used, e.g., via the non-quadratic auxiliary functionals of \cite{chernyshenko2014polynomial,olson2021heat,rosa2021optimal,tobasco2018optimal},
the conjecture that global optimizers of certain chaotic evolutions
are steady \cite{goluskin2019bounds}, or the maximum principle which
is known to improve some bounds \cite{nobili2017limitations,otto2011rayleigh}.
Recently, progress with internally heated convection has led to asymptotically sharp bounds on heat transfer between a pair of steady and perfectly balanced sources and sinks \cite{miquel2019convection}. Due to the balanced nature of the source-sink distribution, which puts in just as much heat as it takes away, the overall rate of heat transfer in the bound turns out to be much larger than the usual ``ultimate scaling'' law; it is nonetheless asymptotically sharp and is saturated by a bulk, convection roll-like flow.

Motivated by these questions, we consider the design of incompressible flows to optimally cool an internally heated disc. We treat a general space- and time-dependent source-sink distribution, under an assumption that more heat is added than is removed. This ensures that some heat must make its way from the bulk to the boundary, and turns out to drive the formation of fine flow microstructure upon optimization (as shown in the upcoming \prettyref{fig:flows}).
Although temperature will be governed in our setup by the usual advection-diffusion equation,
we will not impose a momentum equation directly on the velocity, but
will instead replace it with a constraint on the mean enstrophy it
implies. Optimizing with respect to this constraint leads to a self-similar,
tree-like ``branching flow'' whose cooling we prove is within a
logarithm of globally optimal. While the \emph{a priori} bound part of this result can be shown using the background method,
we prefer a different approach that highlights the variational structure of the advection-diffusion equation. 
We also comment on the related energy-constrained problem, for which considerably less is known.

Setting up these problems in detail, we let $D\subset\mathbb{R}^{2}$
be a disc with radius $R>0$ and consider a temperature field $T(\mathbf{x},t)$
that is passively advected and diffused according to a divergence-free
velocity field $\mathbf{u}(\mathbf{x},t)$ we take to be under our
control. At the disc's boundary $\partial D$ we set the temperature
to be a constant, say $T=0$, and let $f(\mathbf{x},t)$ be a given source-sink distribution. Altogether, $T$ solves the inhomogeneous advection-diffusion
equation
\begin{equation}
\begin{cases}
\partial_{t}T+\mathbf{u}\cdot\nabla T=\kappa\Delta T+f & \text{in }D\\
T=0 & \text{at }\partial D
\end{cases}\label{eq:adv-diff}
\end{equation}
where $\kappa>0$ is the thermal diffusivity. We leave the value of
$T$ at the initial time $t=0$ unspecified, as in the long run it
is lost to diffusion. To measure cooling efficiency, we use the mean-square
temperature gradient\footnote{All limits are understood in the limit superior sense to ensure they
are well-defined. } 
\[
\left\langle |\nabla T|^{2}\right\rangle =\lim_{\tau\to\infty}\,\frac{1}{\tau|D|}\int_{0}^{\tau}\int_{D}|\nabla T(\mathbf{x},t)|^{2}\,d\mathbf{x}dt
\]
whose value depends on $\mathbf{u}$. The notation $|D|$ stands for
the area of the disc. Besides the divergence-free constraint
\[
\nabla\cdot\mathbf{u}=0\quad\text{in }D,
\]
we enforce the no-penetration boundary conditions 
\[
\mathbf{u}\cdot\hat{\mathbf{n}}=0\quad\text{at }\partial D
\]
where $\hat{\mathbf{n}}$ is the outwards pointing unit normal. 

Given this setup, we minimize $\left\langle |\nabla T|^{2}\right\rangle $
amongst all velocities $\mathbf{u}$ with a given value of mean enstrophy
$\left\langle |\nabla\mathbf{u}|^{2}\right\rangle $, or mean kinetic
energy $\left\langle |\mathbf{u}|^{2}\right\rangle $. These ``enstrophy-''
and ``energy-constrained optimal cooling problems'' were originally
posed in \cite{marcotte2018optimal} and analyzed further in \cite{iyer2021bounds}.
Actually, those papers focused on the special case of a constant heat
source, e.g., $f=1$, and were also mostly concerned with the energy-constrained
problem. Here we treat a general source function $f(\mathbf{x},t)$ which may
vary both in space and in time, and may even be negative in some places
and at some times allowing for temporary sinks. Before coming to our
results, we pause to discuss other possible objective functionals
as well as the practical meaning of ``optimizing over $\mathbf{u}$''. 

Various minimization problems have been proposed to optimize cooling.
In \cite{marcotte2018optimal}, the authors minimize the mean temperature
$\left\langle T\right\rangle $ subject to an energy-constraint and for 
a constant positive source. In \cite{iyer2021bounds}, the same steady minimization 
is studied alongside a variety of others involving $L^{p}$-based
quantities, namely, $\left\langle T^{p}\right\rangle$ for $1\leq p<\infty$ and
$\max\,T$ for $p=\infty$. Multiplying the advection-diffusion equation
\prettyref{eq:adv-diff} by $T$ and integrating by parts shows that
\[
\kappa\left\langle |\nabla T|^{2}\right\rangle =\left\langle fT\right\rangle
\]
in general. If $f$ is constant and positive, minimizing $\left\langle T\right\rangle $ is
the same as minimizing $\left\langle |\nabla T|^{2}\right\rangle $.
If $f$ is bounded and uniformly positive, meaning
that $C\geq f\geq c$ for some fixed $c,C\in(0,\infty)$, then 
\[
c\left\langle T\right\rangle \leq\kappa\left\langle |\nabla T|^{2}\right\rangle \leq C\left\langle T\right\rangle 
\]
so that the minimizations give comparable results. For a general and possibly sign-indefinite $f$, 
such a simple relationship need not hold. 
Our choice to minimize $\left\langle |\nabla T|^{2}\right\rangle$ as opposed to, say, $\left\langle |T|^{p}\right\rangle$ 
is partially out of mathematical convenience, but also because the former provides a
more direct assessment of long-term diffusive transport\textemdash the
rate-limiting step in any heat exchange. 

Turning to our choice to treat the velocity directly as a control,
perhaps it is more reasonable from the
viewpoint of applications to
think of controlling an applied force $\mathbf{f}(\mathbf{x},t)$.
One may then seek to optimize cooling subject to a constraint on the
power used. Or, one might limit the complexity of the force, e.g.,
by constraining $\left\langle |\nabla\mathbf{f}|^{2}\right\rangle $.
In any case, the velocity $\mathbf{u}$ will
solve the forced, incompressible Navier-Stokes equations
\begin{equation}
\partial_{t}\mathbf{u}+\mathbf{u}\cdot\nabla\mathbf{u}=-\nabla p+\nu\Delta\mathbf{u}+\mathbf{f}\quad\text{in }D\label{eq:Navier-Stokes}
\end{equation}
where $p(\mathbf{x},t)$ and $\nu>0$ are the pressure and fluid viscosity.
Dotting \prettyref{eq:Navier-Stokes} by $\mathbf{u}$, integrating
by parts, and remembering that $\nabla\cdot\mathbf{u}=0$ yields the
fundamental ``balance law''
\[
\nu\left\langle |\nabla\mathbf{u}|^{2}\right\rangle =\left\langle \mathbf{f}\cdot\mathbf{u}\right\rangle .
\]
Thus, a constraint on the mean power $\left\langle \mathbf{f}\cdot\mathbf{u}\right\rangle $
is the same as a constraint on the mean enstrophy $\left\langle |\nabla\mathbf{u}|^{2}\right\rangle $.
Having found an optimal $\mathbf{u}$, we can simply read off from
\prettyref{eq:Navier-Stokes} the corresponding optimal $\mathbf{f}$.
This is not to say that every optimal velocity is dynamically stable\textemdash in
fact, in related problems involving extremal orbits the opposite situation
holds \cite{goluskin2018bounding,lakshmi2020finding}. Nevertheless,
to get started, we set aside the momentum equation and focus solely
on optimizing advection-diffusion. 

Switching to non-dimensional variables, let $\kappa=R=1$. The value
of the mean enstrophy $\left\langle |\nabla\mathbf{u}|^{2}\right\rangle $,
or the mean energy $\left\langle |\mathbf{u}|^{2}\right\rangle $
when we discuss it, will be called the Pecl\'et number $\text{Pe}$ as either
quantity sets the relative strength of advection to diffusion. (In
dimensional variables, $\text{Pe}=UR/\kappa$ with $U$ being the imposed
velocity scale.) In the advective limit $\text{Pe}\to\infty$, one
expects to be able to drive $\left\langle |\nabla T|^{2}\right\rangle \to0$
along a well-chosen sequence of velocities. The question is: at what
optimal rate can this convergence occur? Our main result is a set
of upper and lower bounds on $\min\,\left\langle |\nabla T|^{2}\right\rangle $
subject to the enstrophy constraint $\left\langle |\nabla\mathbf{u}|^{2}\right\rangle =\text{Pe}^{2}$
that identify the optimal cooling rate up to a possible logarithm
in $\text{Pe}$. Let $\lambda_{1}$ be the first Dirichlet eigenvalue of
$-\Delta$ on $D$, and let $t\wedge s$ denote the minimum of $t,s\in\mathbb{R}$.
\begin{thm}
\label{thm:main-theorem}Let $f(\mathbf{x},t)$ satisfy 
\[
\lim_{\tau\to\infty}\,\frac{1}{\sqrt{\tau}}\int_{0}^{\tau}e^{-\lambda_{1}\left((\tau-t)\wedge t\right)}||f(\cdot,t)||_{L^{2}(D)}\,dt=0\quad\text{and}\quad\left\langle |f|^{2}+|\nabla f|^{2}+|\nabla\nabla f|^{2}\right\rangle <\infty
\]
and assume $\left\langle f\right\rangle >0$. There exist positive
constants $C(f)$, $C'(f)$, and $c(f)$ depending only on $f$ such
that
\[
C\frac{1}{\emph{Pe}^{2/3}}\leq\min_{\substack{\mathbf{u}(\mathbf{x},t)\\
\left\langle |\nabla\mathbf{u}|^{2}\right\rangle =\emph{Pe}^{2}\\
\mathbf{u}\cdot\hat{\mathbf{n}}=0\text{ at }\partial D
}
}\,\left\langle |\nabla T|^{2}\right\rangle \leq C'\frac{\log^{4/3}\emph{Pe}}{\emph{Pe}^{2/3}}
\]
whenever $\emph{Pe}\geq c(f)$. 
\end{thm}
\begin{rem}
The first assumption concerns the possibility that $f$ grows as $t\to\infty$,
and limits that growth such that $\left\langle |\nabla T|^{2}\right\rangle $
remains finite. Without it, one can make $\left\langle |\nabla T|^{2}\right\rangle =\infty$
for all $\mathbf{u}$, e.g., with $f=a(t)\varphi_{1}(\mathbf{x})$
and where $\varphi_{1}\in H_{0}^{1}(D)$ has $-\Delta\varphi_{1}=\lambda_{1}\varphi_{1}$.
The remaining assumptions on $f$ enter into our bounds on $C$ and
$C'$: we achieve $C\gtrsim\left\langle f\right\rangle ^{2}$ and
$C'\lesssim\left\langle |f|^{2}+|\nabla f|^{2}+|\nabla\nabla f|^{2}\right\rangle $.
For the dependence of $c$, see \prettyref{sec:Variational-bounds}
where we prove the lower bound. 
\end{rem}
\begin{rem}
The same result holds if the no-penetration boundary conditions $\mathbf{u}\cdot\hat{\mathbf{n}}=0$
are replaced by the more restrictive no-slip conditions $\mathbf{u}=\mathbf{0}$
at $\partial D$. This is because we use no-slip velocities for
the upper bound. That the lower bound goes through to the no-slip
case is clear. 
\end{rem}
Let us briefly discuss the strategy of our proof. On the one hand,
we must explain why no enstrophy-constrained velocity having $\left\langle |\nabla\mathbf{u}|^{2}\right\rangle =\text{Pe}^{2}$ can lower $\left\langle |\nabla T|^{2}\right\rangle $
significantly beyond $\text{Pe}^{-2/3}$. At the same time, we must construct
a sequence of admissible velocities to cool within a logarithm of
this bound. The challenge is to find a way
of computing heat transfer that at the same time shows how it can be optimized. We follow the
approach of our previous articles on maximizing transport across an
externally heated fluid layer \cite{tobasco2019optimal,souza2020wall,tobasco2017optimal},
the key difference being the presence of the internal source $f(\mathbf{x},t)$.
We show it is possible to bound $\left\langle |\nabla T|^{2}\right\rangle $,
and in fact to compute it exactly in the steady case, by a pair of
variational problems. These problems remind of Dirichlet's principle
for Poisson's equation, but are non-local due to the mixed character
of the operator $\mathbf{u}\cdot\nabla-\Delta$. That there exists
a non-local Dirichlet-like principle for the effective diffusivity
of a periodic or random fluid flow is well-known in homogenization
\cite{avellaneda1991integral,fannjiang1994convection}. 
Our idea is that $\left\langle |\nabla T|^{2}\right\rangle $ should behave like
an effective diffusivity upon optimization. This is very much in the
spirit of the general connection between homogenization and optimal
design (see, e.g., \cite{kohn1986optimal_i,kohn1986optimal_ii,kohn1986optimal_iii}).

\begin{figure}
\includegraphics[width=\textwidth]{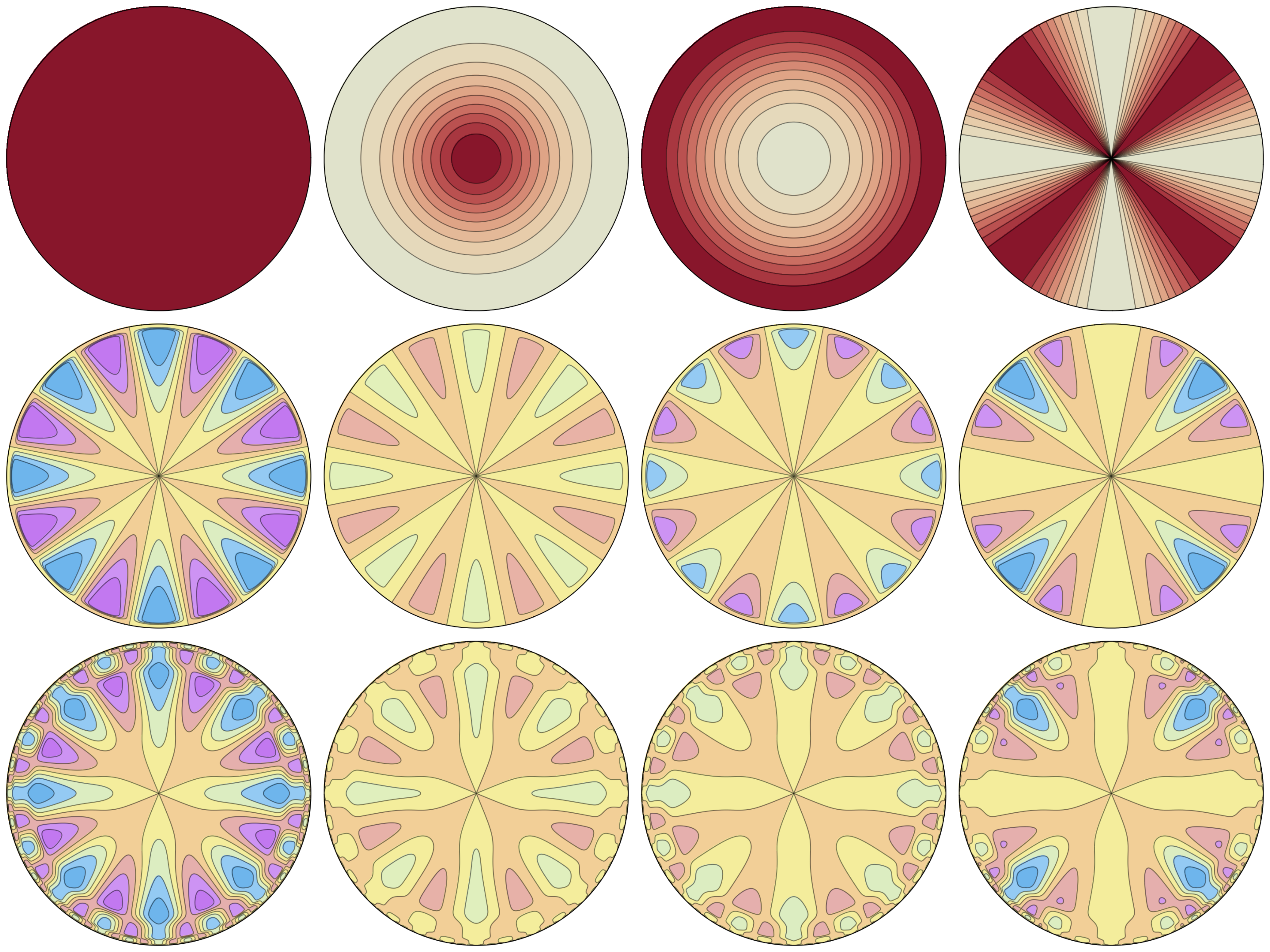}
\caption{Divergence-free flow designs for cooling an internally heated unit disc. The top row shows contours of the source functions $f=1$, $e^{-4r^2}$, $e^{-4(1-r)^2}$, and $\sin^2(2\theta)$ from left to right. The middle row shows streamlines of convection roll-like designs adapted to these sources. The bottom row shows their branching flow counterparts. Counterclockwise/clockwise circulations are colored purple/blue.}\label{fig:flows}
\end{figure}

The bulk of our work goes into constructing the ``branching flows''
behind the upper bound in \prettyref{thm:main-theorem}.  In general,
we envision an unsteady, tree-like, multi-scale flow whose features
refine from the bulk to the boundary as in the bottom of \prettyref{fig:flows}. To achieve it, we piece together
a family of convection roll-like flows, with the number of rolls
in the azimuthal direction $\theta$ depending on the radial coordinate
$r$. A similar, albeit steady, branching flow was used in \cite{tobasco2019optimal,tobasco2017optimal}
to prove nearly sharp bounds on optimal transport through a fluid layer.
Later on, a beautiful and fully three-dimensional branching flow
was found numerically via gradient ascent \cite{motoki2018maximal}.
Each of these may be regarded as a more refined version of Busse's ``multi-$\alpha$''
flows \cite{busse1969howards}, the latter of which do not enforce the
full advection-diffusion equation but instead hinge upon various balance
laws it implies. 

Branching may be anticipated as a mechanism for optimal heat transfer
by asking what it takes to guide Brownian particles from the bulk to the boundary of the disc. 
Imagine a cloud of particles is released
at a radius $r$ and is to be transported outwards across a distance $\sim 1-r$, 
say with (non-dimensional) speed $\sim U$. In time $\tau$, the particles diffuse in the $\theta$-direction
by a typical distance $\sim\sqrt{\tau}$. In the same time, they advect
in the $r$-direction by a distance $\sim U\cdot\tau$. With the goal
of not losing too many particles to poorly directed advection, we
suggest using streamlines whose azimuthal scale $\ell(r)$ matches
that of the Brownian cloud: 
\begin{equation}
\ell(r)\sim\sqrt{\frac{1-r}{U}}.\label{eq:azimuthal-scale}
\end{equation}
Two parameters emerge, namely the speed $U$ and the boundary layer
width $\delta_{\text{bl}}$ where the streamlines finally turn
around. By the enstrophy constraint,
\[
\text{Pe}^{2}=\left\langle |\nabla\mathbf{u}|^{2}\right\rangle \sim\int^{1-\delta_{\text{bl}}}\left(\frac{U}{\ell(r)}\right)^{2}\,rdr\sim U^{3}\log\frac{1}{\delta_{\text{bl}}}.
\]
Thinking of an isotropic roll-based boundary layer, we take $\sqrt{\delta_{\text{bl}}/U}\sim\delta_{\text{bl}}$
giving $U\sim1/\delta_{\text{bl}}$. This determines all parameters
via the scalings
\begin{equation}
U\sim\frac{\text{Pe}^{2/3}}{\log^{1/3}\text{Pe}}\quad\text{and}\quad\delta_{\text{bl}}\sim\frac{\log^{1/3}\text{Pe}}{\text{Pe}^{2/3}}.\label{eq:scalings-set}
\end{equation}

Even if our thought experiment on how advection can effectively ``hug''
diffusion suggests a particular way of branching, it is far from a
proof of \prettyref{thm:main-theorem}, or even of its upper bound.
To achieve it, we apply our variational bounds to estimate $\left\langle |\nabla T|^{2}\right\rangle $
on a general roll-based branching flow design. By optimizing over
all admissible ``scale functions'' $\ell(r)$, we discover the very
same scalings as in \prettyref{eq:azimuthal-scale} and \prettyref{eq:scalings-set}.
In fact, we see no way out of these and the logarithmically corrected
bound they imply: we conjecture that 
\[
\min_{\substack{\mathbf{u}(\mathbf{x},t)\\
\left\langle |\nabla\mathbf{u}|^{2}\right\rangle =\text{Pe}^{2}\\
\mathbf{u}\cdot\hat{\mathbf{n}}=0\text{ at }\partial D
}
}\,\left\langle |\nabla T|^{2}\right\rangle =C(f)\cdot\frac{\log^{4/3}\text{Pe}}{\text{Pe}^{2/3}}+\cdots\quad\text{as }\text{Pe}\to\infty
\]
with the dots representing asymptotically negligible terms. This is
for two-dimensional cooling; the data in \cite{motoki2018maximal}
suggest a pure power law scaling for optimal transport across a three-dimensional
layer. Maybe there is more room for chaperoning particles in three
dimensions versus two? In any case, proving the analog of our conjecture
for a fluid layer would establish a new bound on the longstanding
problem of Rayleigh-B\'enard convection, as noted in \cite{tobasco2019optimal}.
This bound would state that $\text{Nu}\lesssim \text{Ra}^{1/2}/\log^{2}\text{Ra}$ with
$\text{Nu}$ and $\text{Ra}$ being the Nusselt and Rayleigh numbers (the relevant
measures of transport and driving there).

Stochastic analysis of optimal cooling is the focus of \cite{iyer2021bounds}
where the authors estimate, amongst other things, the mean temperature
of steady convection roll-like flows in an internally heated layer with a constant source. 
Using large deviations-based bounds, they prove that $\min\,\left\langle T\right\rangle \leq c(\delta)\frac{\log \text{Pe}}{\text{Pe}^{1-\delta}}$
for any $\delta>0$, where $\left\langle |\mathbf{u}|^{2}\right\rangle =\text{Pe}^{2}$. 
Showing the conjectured optimal scaling 
\[
\min_{\substack{\mathbf{u}(\mathbf{x},t)\\
\left\langle |\mathbf{u}|^{2}\right\rangle =\text{Pe}^{2}\\
\mathbf{u}\cdot\hat{\mathbf{n}}=0\text{ at }\partial D
}
}\,\left\langle |\nabla T|^{2}\right\rangle =C(f)\cdot\frac{1}{\text{Pe}}+\cdots\quad\text{as }\text{Pe}\to\infty
\]
put forth for $f=1$ in \cite{marcotte2018optimal} remains an open challenge, the lower bound part of which seems to require a new idea.  
While the arguments behind \prettyref{thm:main-theorem} lead without major modifications to the supposedly sub-optimal lower bound $\min\,\left\langle |\nabla T|^2\right\rangle \geq C(f)/\text{Pe}^2$, 
they are strong enough to show the conjectured upper bound $\min\,\left\langle |\nabla T|^2\right\rangle \leq C'(f)/\text{Pe}$ and to likewise improve the estimate from \cite{iyer2021bounds}. 
The $\text{Pe}^{-1}$ scaling was found in \cite{marcotte2018optimal} through a matched asymptotic analysis of convection roll-like flows, but without a bound on the error terms. 
See \prettyref{prop:lower-bound-energy-constraint} and \prettyref{prop:upper-bound-energy-constraint}
for our lower and upper bounds on energy-constrained cooling.

This paper is organized as follows. We begin in \prettyref{sec:Variational-bounds}
with our variational bounds on cooling. At the end of that section,
we choose a background-like test function to prove our lower bounds on the enstrophy- and energy-constrained problems, including the lower bound part of  \prettyref{thm:main-theorem}.
The rest of the paper is devoted to upper bounds. 
First, we discuss the steady enstrophy-constrained problem in \prettyref{sec:Steady-branching-flows}
to explain our strategy for finding designs. Then, in \prettyref{sec:Unsteady-branching-flows}
we construct our branching flows to prove the upper bound part of \prettyref{thm:main-theorem}.
We end in \prettyref{sec:Unsteady-roll-type-flows-energy-constrained}
with our upper bound on energy-constrained cooling.

\subsection{Notation\label{subsec:Notation}}

We write $X\lesssim Y$ if $X\leq cY$ for a fixed numerical constant
$c>0$, i.e., one that is independent of all parameters. We write
$X\lesssim_{a}Y$ if $X\leq cY$ where $c=c(a)$, and $X\ll Y$ if
$\frac{X}{Y}\to0$ in a relevant limit. We abbreviate $X\wedge Y=\min\left\{ X,Y\right\} $
and $X\vee Y=\max\left\{ X,Y\right\} $. 

We often conflate a point $\mathbf{x}$ with its polar coordinates
$(r,\theta)$. The unit vectors $\hat{\mathbf{e}}_{r}=\mathbf{x}/|\mathbf{x}|$
and $\hat{\mathbf{e}}_{\theta}=\hat{\mathbf{e}}_{r}^{\perp}$, where
$(\cdot)^{\perp}$ is a counterclockwise rotation by $\pi/2$. Given
a function $\varphi(\mathbf{x})$, its average on the disc $D$ is
\[
\fint_{D}\varphi=\frac{1}{|D|}\int_{D}\varphi(\mathbf{x})\,d\mathbf{x}
\]
where $|D|$ is the disc's area. By $D_{r}$ we mean the concentric
disc of radius $r$. Restricting $\varphi$ to $\partial D_{r}$ gives
a function of $\theta$, whose average and $L^{1}$-norm are
\[
\overline{\varphi}(r)=\frac{1}{2\pi}\int_{0}^{2\pi}\varphi(r,\theta)\,d\theta\quad\text{and}\quad||\varphi||_{L_{\theta}^{1}}(r)=\int_{0}^{2\pi}|\varphi(r,\theta)|\,d\theta.
\]
Given a function $\varphi(\mathbf{x},t)$, we write 
\[
\left\langle \varphi\right\rangle =\limsup_{\tau\to\infty}\,\fint_{0}^{\tau}\fint_{D}\varphi=\limsup_{\tau\to\infty}\,\frac{1}{\tau}\frac{1}{|D|}\int_{0}^{\tau}\int_{D}\varphi(\mathbf{x},t)\,d\mathbf{x}dt
\]
for its (limit superior) space and long-time average.

The Sobolev spaces $L^{2}(D)$ and $H^{1}(D)$ are defined as usual,
using the norms
\[
||\varphi||_{L^{2}(D)}=\sqrt{\int_{D}|\varphi|^{2}\,d\mathbf{x}}\quad\text{and}\quad||\varphi||_{H^{1}(D)}=\sqrt{\int_{D}|\varphi|^{2}+|\nabla\varphi|^{2}\,d\mathbf{x}}.
\]
We write $H^{-1}(D)$ for the dual of $H_{0}^{1}(D)$, the latter
indicating $H^{1}$-functions on $D$ with zero trace at $\partial D$.
It will be convenient to normalize their duality bracket $(\cdot,\cdot)$
by $|D|$, and in particular we take
\[
(\nabla\cdot\mathbf{m},\varphi)=-\fint_{D}\mathbf{m}\cdot\nabla\varphi
\]
for $\mathbf{m}\in L^{2}(D;\mathbb{R}^{2})$ and $\varphi\in H_{0}^{1}(D)$.
We use various mixed spaces such as $L^{2}((0,\infty);H_{0}^{1}(D))$
and its local-in-time version $L_{\text{loc}}^{2}((0,\infty);H_{0}^{1}(D))$.
For definitions, see a text on partial differential equations (PDE)
such as \cite{evans2010partial}. 

\subsection{Acknowledgements}

We thank Charlie Doering for many inspiring discussions, and David
Goluskin for noting the relevance of balanced versus unbalanced heating.
This work was supported by National Science Foundation Award DMS-2025000.

\section{Variational bounds on cooling \label{sec:Variational-bounds}}

We begin with a general method for bounding $\left\langle |\nabla T|^{2}\right\rangle $
from above and below. While it is easy to show using an integration
by parts that 
\begin{equation}
\left\langle |\nabla T|^{2}\right\rangle \lesssim\left\langle |f|^{2}\right\rangle ,\label{eq:easy-bound}
\end{equation}
improving this upper bound and finding a corresponding lower bound
is not so simple. Incidentally, it follows from \prettyref{eq:easy-bound}
and the linearity of the PDE in $T$ that $\left\langle |\nabla T|^{2}\right\rangle $
does not depend on the exact choice of its initial data, so long as
it belongs to $L^{2}(D)$. To obtain better bounds, we invoke a certain
symmetrizing change of variables that couples $T$ to a second ``temperature''
field arising from its adjoint PDE. The resulting bounds are sharp
in the steady case where $\mathbf{u}$ and $f$ do not depend on time. 

Define the admissible set
\[
\mathcal{A}=\left\{ \theta\in L_{\text{loc}}^{2}((0,\infty);H_{0}^{1}(D)):\partial_{t}\theta\in L_{\text{loc}}^{2}((0,\infty);H^{-1}(D)),\ \limsup_{\tau\to\infty}\,\frac{1}{\sqrt{\tau}}||\theta(\cdot,\tau)||_{L^{2}(D)}<\infty\right\} 
\]
and let $\Delta^{-1}:H^{-1}(D)\to H_{0}^{1}(D)$ denote the inverse Laplacian with zero Dirichlet boundary conditions. Recall 
\[
\lambda_{1}=\min_{\varphi\in H_{0}^{1}(D)}\,\frac{\int_{D}|\nabla\varphi|^{2}}{\int_{D}|\varphi|^{2}}.
\]

\begin{prop}
\label{prop:aprioribds}Let $\mathbf{u}(\mathbf{x},t)$ be weakly
divergence-free and have 
\[
\left\langle |\mathbf{u}|^{2}\right\rangle <\infty
\]
and let $f(\mathbf{x},t)$ satisfy
\[
\lim_{\tau\to\infty}\,\frac{1}{\sqrt{\tau}}\int_{0}^{\tau}e^{-\lambda_{1}\left((\tau-t)\wedge t\right)}||f(\cdot,t)||_{L^{2}(D)}\,dt=0.
\]
Any weak solution $T(\mathbf{x},t)$ of 
\[
\begin{cases}
\partial_{t}T+\mathbf{u}\cdot\nabla T=\Delta T+f & \text{in }D\\
T=0 & \text{at }\partial D
\end{cases}
\]
with initial data $T(\cdot,0)\in L^{2}(D)$ must obey the bounds 
\[
\left\langle 2\xi f-\left|\nabla\xi\right|^{2}-\left|\nabla\Delta^{-1}(\partial_{t}+\mathbf{u}\cdot\nabla)\xi\right|^{2}\right\rangle \leq\left\langle |\nabla T|^{2}\right\rangle \leq\left\langle \left|\nabla\eta\right|^{2}+\left|\nabla\Delta^{-1}\left[(\partial_{t}+\mathbf{u}\cdot\nabla)\eta-f\right]\right|^{2}\right\rangle 
\]
for all $\xi,\eta\in\mathcal{A}$.
\end{prop}
\begin{rem}
\label{rem:weak-solution-clarification}The statement that $T$ is
a weak solution deserves to be clarified, especially as $\mathbf{u}\cdot\nabla T$
is at first glance only in $L^{1}(D)$, a.e.\ in time, and so would
appear not to belong to $H^{-1}(D)=(H_{0}^{1}(D))'$. (Two dimensions
is critical for the relevant Sobolev embedding.) However, $\mathbf{u}$ and $\nabla T$
are divergence- and curl-free, and this is enough to compensate for
their lack of regularity. 

To see why, introduce a stream function $\psi\in H^{1}(D)$ with $\mathbf{u}=\nabla^{\perp}\psi$,
and note that $\mathbf{u}\cdot\nabla T$ is the Jacobian determinant
of the mapping $\mathbf{x}\mapsto(\psi,T)$. Therefore, by the estimate
of Coifman, Lions, Meyer, and Semmes on Jacobian determinants \cite{coifman1993compensated}
(see also \cite{muller1994hardy}), 
\[
\left|\int_{D}\mathbf{u}\cdot\nabla T\varphi\,d\mathbf{x}\right|\lesssim||\mathbf{u}||_{L^{2}(D)}||\nabla T||_{L^{2}(D)}||\nabla\varphi||_{L^{2}(D)}\quad\forall\,\varphi\in C_{c}^{\infty}(D).
\]
Letting $(\cdot,\cdot)$ denote the (normalized) duality bracket between
$H^{-1}(D)$ and $H_{0}^{1}(D)$, the previous inequality extends
to say that 
\begin{equation}
\left|\left(\mathbf{u}\cdot\nabla T,\varphi\right)\right|\lesssim||\mathbf{u}||_{L^{2}(D)}||\nabla T||_{L^{2}(D)}||\nabla\varphi||_{L^{2}(D)}\quad\forall\,\varphi\in H_{0}^{1}(D).\label{eq:compensated-bound}
\end{equation}
Thus, $\mathbf{u}\cdot\nabla T\in H^{-1}(D)$ at a.e.\ time. Clearly,
the same holds for $\Delta T$ and $f$. 

At this point, the definition of $T$ as a weak solution can be carried
out as usual, by requiring that 
\[
T\in L_{\text{loc}}^{2}((0,\infty);H_{0}^{1}(D))\quad\text{and}\quad\partial_{t}T\in L_{\text{loc}}^{2}((0,\infty);H^{-1}(D))
\]
and enforcing the PDE as an equality on $H^{-1}(D)$. Existence and
uniqueness of weak solutions with initial data in $L^{2}(D)$ follows
(see, e.g., \cite{evans2010partial}).
\end{rem}
\begin{rem}
\label{rem:anti-symmetric-integral}A particular consequence of the
previous remark that will be used throughout the proof of \prettyref{prop:aprioribds}
is that 
\[
\left(\mathbf{u}\cdot\nabla T,T\right)=0\quad\text{for a.e. }t
\]
if $\mathbf{u}$ and $T$ are as in the proposition. This follows
from a smooth approximation argument and the fact that $T(\cdot,t)\in H_{0}^{1}(D)$
a.e.\ in $t$. Indeed, if $T(\cdot,t)$ is smooth and compactly supported
in $D$ then
\[
\left(\mathbf{u}\cdot\nabla T,T\right)=\frac{1}{2}\fint_{D}\mathbf{u}\cdot\nabla|T|^{2}\,d\mathbf{x}=0
\]
by the divergence theorem. The bilinear form $T\mapsto(\mathbf{u}\cdot\nabla T,T)$
is continuous on $H_{0}^{1}(D)$, per \prettyref{eq:compensated-bound}.
The claimed identity now follows from the density of smooth and compactly
supported functions in $H_{0}^{1}(D)$.
\end{rem}
\begin{rem}
\label{rem:steady-bds}In the case that $\mathbf{u}$ and $f$ are
steady, the bounds from the proposition become sharp and are achieved
by steady test functions $\eta(\mathbf{x})$ and $\xi(\mathbf{x})$.
That is,
\[
\left\langle |\nabla T|^{2}\right\rangle =\min_{\eta\in H_{0}^{1}(D)}\,\fint_{D}\left|\nabla\eta\right|^{2}+\left|\nabla\Delta^{-1}\left[\mathbf{u}\cdot\nabla\eta-f\right]\right|^{2}\,d\mathbf{x}=\max_{\xi\in H_{0}^{1}(D)}\,\fint_{D}2\xi f-\left|\nabla\xi\right|^{2}-\left|\nabla\Delta^{-1}\mathbf{u}\cdot\nabla\xi\right|^{2}\,d\mathbf{x}.
\]
The crucial point is that, when $\mathbf{u}$ and $f$ are steady,
the above optimizations in $\eta$ and $\xi$ not only provide bounds
on $\left\langle |\nabla T|^{2}\right\rangle $, but also turn out
to be ``strongly'' dual in that their optimal values are the same
(an easy consequence of their Euler-Lagrange equations). See \cite{tobasco2019optimal}
for the full proof of a similar duality arising for steady heat transport
through a fluid layer. 
\end{rem}
\begin{rem}
\label{rem:better-bds}On the other hand, if $f$ is allowed to depend
on time, equality need not hold between our bounds on $\left\langle |\nabla T|^{2}\right\rangle $.
This is because one cannot rule out the possibility that
\[
\liminf_{\tau\to\infty}\,\fint_{0}^{\tau}\fint_{D}|\nabla T|^{2}<\limsup_{\tau\to\infty}\,\fint_{0}^{\tau}\fint_{D}|\nabla T|^{2}
\]
in general while, as will become clear in the proof below, our bounds
actually hold on these limit inferior and limit superior long-time
averages (see \prettyref{eq:better-upperbound} and \prettyref{eq:better-lowerbound}).
We do not know if our bounds are sharp when $f(\mathbf{x})$ is steady
and $\mathbf{u}(\mathbf{x},t)$ is unsteady, though we guess the answer
is ``no''.
\end{rem}
\begin{proof}
Consider the formally adjoint pair of problems
\[
\begin{cases}
(\pm\partial_{t}\pm\mathbf{u}\cdot\nabla)T_{\pm}=\Delta T_{\pm}+f & \text{in }D\\
T_{\pm}=0 & \text{at }\partial D
\end{cases},
\]
the $+$ version of which is solved by the given temperature $T$.
As the value of $\left\langle |\nabla T|^{2}\right\rangle $ does
not depend on its initial data, we let $T(\cdot,0)=0$ and refer to
this version of the temperature as $T_{+}$ in the rest of this proof.
Fixing an arbitrary time $\tau>0$ which we will eventually take to
$\infty$, let $T_{-}$ be the unique weak solution of the $-$ problem
with $T_{-}(\cdot,\tau)=0$. (The definition of ``weak solution''
is standard; see \prettyref{rem:weak-solution-clarification}.) A
quick argument using integration by parts and Gronwall's inequality
shows that
\[
||T_{+}(\cdot,\tau)||_{L^{2}(D)}\vee||T_{-}(\cdot,0)||_{L^{2}(D)}\leq\int_{0}^{\tau}e^{-\lambda_{1}\left((\tau-t)\wedge t\right)}||f(\cdot,t)|_{L^{2}(D)}\,dt\ll\sqrt{\tau}\quad\text{as }\tau\to\infty
\]
by our hypothesis on $f$. Briefly, one notes for $Z=\sqrt{\epsilon+||T_{+}||_{L^{2}(D)}^{2}}$
that 
\[
\frac{d}{dt}Z+\lambda_{1}Z\leq\lambda_{1}\epsilon^{1/2}+||f||_{L^{2}(D)}
\]
for all $\epsilon>0$. Using the integrating factor $e^{-\lambda_{1}(\tau-t)}$
and taking $\epsilon\to0$ yields the $+$ version of the inequality.
The $-$ version follows by reversing time. 

Having defined $T_{\pm}$, we now change variables to the pair
\[
\eta=\frac{1}{2}(T_{+}-T_{-})\quad\text{and}\quad\xi=\frac{1}{2}(T_{+}+T_{-})
\]
and note they satisfy
\begin{equation}
\begin{cases}
(\partial_{t}+\mathbf{u}\cdot\nabla)\eta=\Delta\xi+f & \text{in }D\\
(\partial_{t}+\mathbf{u}\cdot\nabla)\xi=\Delta\eta & \text{in }D\\
\eta=\xi=0 & \text{at }\partial D
\end{cases}\label{eq:mixed-PDEs}
\end{equation}
weakly for $t\in(0,\tau)$. It follows from our previous bounds on
$T_{\pm}$ that
\begin{equation}
||\xi(\cdot,0)||_{L^{2}(D)}+||\xi(\cdot,\tau)||_{L^{2}(D)}+||\eta(\cdot,0)||_{L^{2}(D)}+||\eta(\cdot,\tau)||_{L^{2}(D)}\ll\sqrt{\tau}\quad\text{as }\tau\to\infty.\label{eq:Linftybds-on-mixed-vars}
\end{equation}
Introduce the notation $\left\langle \cdot\right\rangle _{\tau}=\fint_{0}^{\tau}\fint_{D}\cdot\,d\mathbf{x}dt$
for the truncated space- and time-average. Testing the second PDE
in \prettyref{eq:mixed-PDEs} against $\xi$, we get that 
\begin{align*}
\left\langle \nabla\eta\cdot\nabla\xi\right\rangle _{\tau} & =\fint_{0}^{\tau}(-\Delta\eta,\xi)\,dt=\fint_{0}^{\tau}(-(\partial_{t}+\mathbf{u}\cdot\nabla)\xi,\xi)\,dt\\
 & =\fint_{0}^{\tau}\left[-\frac{d}{dt}\frac{1}{2}\fint_{D}|\xi|^{2}\,d\mathbf{x}+(\mathbf{u}\cdot\nabla\xi,\xi)\right]\,dt\\
 & =\frac{1}{2\tau}\fint_{D}|\xi(\mathbf{x},0)|^{2}-|\xi(\mathbf{x},\tau)|^{2}\,d\mathbf{x}
\end{align*}
where again $(\cdot,\cdot)$ is the duality bracket normalized by
$|D|$. That $\mathbf{u}\cdot\nabla\xi\in H^{-1}(D)$ at a.e.\ time
follows from the assumed mean-square integrability of $\mathbf{u}$
and the statement that it is weakly divergence-free; see \prettyref{rem:weak-solution-clarification}
and \prettyref{rem:anti-symmetric-integral} for the proof that $(\mathbf{u}\cdot\nabla\xi,\xi)=0$.
Combined with \prettyref{eq:Linftybds-on-mixed-vars}, the conclusion
is that
\[
\left\langle \nabla\eta\cdot\nabla\xi\right\rangle _{\tau}\to0\quad\text{as }\tau\to\infty.
\]
Hence, $T=T_{+}=(\eta+\xi)/2$ obeys
\begin{align}
\left\langle |\nabla T_{+}|^{2}\right\rangle _{\tau} & =\left\langle |\nabla\eta|^{2}+|\nabla\xi|^{2}\right\rangle _{\tau}+o_{\tau}(1)\label{eq:first-identity}\\
 & =\left\langle |\nabla\eta|^{2}+|\nabla\Delta^{-1}\left[(\partial_{t}+\mathbf{u}\cdot\nabla)\eta-f\right]|^{2}\right\rangle _{\tau}+o_{\tau}(1)\nonumber 
\end{align}
with $o_{\tau}(1)$ denoting a quantity that vanishes as $\tau\to\infty$.
This way of writing the mean-square temperature gradient leads to
the upper bound from the claim. 

To see why, let $\tilde{\eta}\in\mathcal{A}$ and consider the difference
\[
A_{\tau}:=\left\langle |\nabla\tilde{\eta}|^{2}+\left|\nabla\Delta^{-1}\left[(\partial_{t}+\mathbf{u}\cdot\nabla)\tilde{\eta}-f\right]\right|^{2}\right\rangle _{\tau}-\left\langle |\nabla\eta|^{2}+\left|\nabla\Delta^{-1}\left[(\partial_{t}+\mathbf{u}\cdot\nabla)\eta-f\right]\right|^{2}\right\rangle _{\tau}.
\]
We claim it is non-negative up to a term that vanishes as $\tau\to\infty$.
Since $|\cdot|^{2}$ is convex, 
\begin{align*}
\frac{1}{2}A_{\tau} & \geq\left\langle \nabla\eta\cdot\nabla(\tilde{\eta}-\eta)+\nabla\xi\cdot\nabla\Delta^{-1}(\partial_{t}+\mathbf{u}\cdot\nabla)(\tilde{\eta}-\eta)\right\rangle _{\tau}\\
 & =\fint_{0}^{\tau}\left[\fint_{D}\nabla\eta\cdot\nabla(\tilde{\eta}-\eta)\,d\mathbf{x}-((\partial_{t}+\mathbf{u}\cdot\nabla)(\tilde{\eta}-\eta),\xi)\right]\,dt\\
 & =\fint_{0}^{\tau}(-\Delta\eta+(\partial_{t}+\mathbf{u}\cdot\nabla)\xi,\tilde{\eta}-\eta)\,dt-\frac{1}{\tau}\left(\fint_{D}(\tilde{\eta}-\eta)\xi\right)|_{t=0}^{t=\tau}\geq-o_{\tau}(1)
\end{align*}
again by \prettyref{eq:Linftybds-on-mixed-vars} and the growth condition
in our definition of $\mathcal{A}$. Hence,
\begin{equation}
\left\langle |\nabla T_{+}|^{2}\right\rangle _{\tau}\leq\left\langle |\nabla\tilde{\eta}|^{2}+\left|\nabla\Delta^{-1}\left[(\partial_{t}+\mathbf{u}\cdot\nabla)\tilde{\eta}-f\right]\right|^{2}\right\rangle _{\tau}+o_{\tau}(1)\label{eq:better-upperbound}
\end{equation}
for all $\tilde{\eta}\in\mathcal{A}$. This bound is a bit
better than the one from the claim (however, see \prettyref{rem:better-bds}).

Returning to the PDEs in \prettyref{eq:mixed-PDEs}, we now test the
second one against $\eta$ and integrate by parts, again using the
abbreviation $o_{\tau}(1)$ for terms that limit to zero as $\tau\to\infty$.
The result is that
\begin{align*}
\left\langle |\nabla\eta|^{2}\right\rangle _{\tau} & =\fint_{0}^{\tau}(-\Delta\eta,\eta)\,dt=\fint_{0}^{\tau}(-(\partial_{t}+\mathbf{u}\cdot\nabla)\xi,\eta)\,dt\\
 & =\fint_{0}^{\tau}((\partial_{t}+\mathbf{u}\cdot\nabla)\eta,\xi)\,dt+o_{\tau}(1)=\fint_{0}^{\tau}(\Delta\xi+f,\xi)\,dt+o_{\tau}(1)\\
 & =\left\langle -|\nabla\xi|^{2}+f\xi\right\rangle _{\tau}+o_{\tau}(1)
\end{align*}
where in the second line we applied the first PDE in \prettyref{eq:mixed-PDEs}.
Combined with \prettyref{eq:first-identity}, this proves that
\[
\left\langle |\nabla T_{+}|^{2}\right\rangle _{\tau}=\left\langle f\xi\right\rangle _{\tau}+o_{\tau}(1)
\]
and that
\begin{align*}
\left\langle |\nabla T_{+}|^{2}\right\rangle _{\tau} & =\left\langle 2f\xi-|\nabla\eta|^{2}-|\nabla\xi|^{2}\right\rangle _{\tau}+o_{\tau}(1)\\
 & =\left\langle 2f\xi-|\nabla\xi|^{2}-|\nabla\Delta^{-1}(\partial_{t}+\mathbf{u}\cdot\nabla)\xi|^{2}\right\rangle _{\tau}+o_{\tau}(1).
\end{align*}
We are ready to deduce the lower bound from the claim.

Let $\tilde{\xi}\in\mathcal{A}$ and call
\[
B_{\tau}:=\left\langle 2f\xi-|\nabla\xi|^{2}-|\nabla\Delta^{-1}(\partial_{t}+\mathbf{u}\cdot\nabla)\xi|^{2}\right\rangle _{\tau}-\left\langle 2f\tilde{\xi}-|\nabla\tilde{\xi}|^{2}-|\nabla\Delta^{-1}(\partial_{t}+\mathbf{u}\cdot\nabla)\tilde{\xi}|^{2}\right\rangle _{\tau}.
\]
As before, we use the convexity of $|\cdot|^{2}$ to write that 
\begin{align*}
\frac{1}{2}B_{\tau} & \geq\left\langle -f(\tilde{\xi}-\xi)+\nabla\xi\cdot\nabla(\tilde{\xi}-\xi)+\nabla\eta\cdot\nabla\Delta^{-1}(\partial_{t}+\mathbf{u}\cdot\nabla)(\tilde{\xi}-\xi)\right\rangle _{\tau}\\
 & =\fint_{0}^{\tau}\left[\fint_{D}-f(\tilde{\xi}-\xi)+\nabla\xi\cdot\nabla(\tilde{\xi}-\xi)\,dx-((\partial_{t}+\mathbf{u}\cdot\nabla)(\tilde{\xi}-\xi),\eta)\right]\,dt\\
 & =\fint_{0}^{\tau}(-f-\Delta\xi+(\partial_{t}+\mathbf{u}\cdot\nabla)\eta,\tilde{\xi}-\xi)\,dt-\frac{1}{\tau}\left(\fint_{D}(\tilde{\xi}-\xi)\eta\right)|_{t=0}^{t=\tau}\geq-o_{\tau}(1).
\end{align*}
In the last step we applied \prettyref{eq:Linftybds-on-mixed-vars}
to handle the terms at $t=0$ and $t=\tau$. Thus, 
\begin{equation}
\left\langle |\nabla T_{+}|^{2}\right\rangle _{\tau}\geq\left\langle 2f\tilde{\xi}-|\nabla\tilde{\xi}|^{2}-|\nabla\Delta^{-1}(\partial_{t}+\mathbf{u}\cdot\nabla)\tilde{\xi}|^{2}\right\rangle _{\tau}-o_{\tau}(1)\label{eq:better-lowerbound}
\end{equation}
for all $\tilde{\xi}\in\mathcal{A}$. Again, the result is a bit better
than the bound from the claim.
\end{proof}
As a first application of our variational bounds on cooling, we obtain
the lower bound from \prettyref{thm:main-theorem} which applies to
all velocities with a specified mean enstrophy. 

\begin{proof}[Proof of the lower bound from \prettyref{thm:main-theorem}]
The lower bound half of \prettyref{prop:aprioribds} applied with
a steady test function $\xi\in H_{0}^{1}(D)$ shows that
\[
2\left\langle f\xi\right\rangle \leq\left\langle |\nabla T|^{2}\right\rangle +\left\langle |\nabla\xi|^{2}+|\nabla\Delta^{-1}\nabla\cdot(\mathbf{u}\xi)|^{2}\right\rangle .
\]
Scaling $\xi\to\lambda\xi$ and optimizing over $\lambda\in\mathbb{R}$,
we now write that
\begin{equation}
\left\langle f\xi\right\rangle ^{2}\leq\left\langle |\nabla T|^{2}\right\rangle \left\langle |\nabla\xi|^{2}+|\nabla\Delta^{-1}\nabla\cdot(\mathbf{u}\xi)|^{2}\right\rangle .\label{eq:bound-to-be-used}
\end{equation}
This bound holds for all $\xi\in H_{0}^{1}(D)$, and we proceed to
make a choice. 

Given any $\delta\in(0,1)$, define 
\[
\xi_{\delta}=\chi_{\delta}(r)
\]
where $\chi_{\delta}(r)$ is a smooth, radial ``cutoff function''
that goes from zero to one across a small boundary layer. Precisely,
$\chi_{\delta}\in C_{c}^{\infty}([0,1))$ satisfies
\[
0\leq\chi_{\delta}(r)\leq1\ \forall\,r\in[0,1)\quad\text{and}\quad\chi_{\delta}(r)=1\ \forall\,r\leq1-\delta
\]
and has
\[
||\chi_{\delta}'||_{L^{\infty}([0,1))}\lesssim\frac{1}{\delta}
\]
with a constant independent of $\delta$. Note that 
\begin{align*}
\left|\left\langle f\right\rangle \right| & \leq\left|\left\langle f\xi_{\delta}\right\rangle \right|+\left|\left\langle f(\xi_{\delta}-1)\right\rangle \right|\leq\left|\left\langle f\xi_{\delta}\right\rangle \right|+\sqrt{\left\langle |f|^{2}\right\rangle }\sqrt{\left\langle \indicator{r>1-\delta}\right\rangle }\\
 & \leq\left|\left\langle f\xi_{\delta}\right\rangle \right|+C'(f)\sqrt{\delta}.
\end{align*}
Hence, there exists $\delta_{0}(f)>0$ such that 
\begin{equation}\label{eq:some-f}
\frac{1}{2}\left|\left\langle f\right\rangle \right|\leq\left|\left\langle f\xi_{\delta}\right\rangle \right|\quad\forall\,\delta\in(0,\delta_{0}].
\end{equation}
Similarly, 
\begin{equation}\label{eq:divergent-gradient-bd}
\left\langle |\nabla\xi_{\delta}|^{2}\right\rangle =\fint_{D}|\chi'_{\delta}(r)|^{2}\lesssim\fint_{D}\frac{1}{\delta^{2}}\indicator{r>1-\delta}\lesssim\frac{1}{\delta}
\end{equation}
for all $\delta$. 

At this point, we have dealt with each of the terms in the bound \prettyref{eq:bound-to-be-used}
except for the non-local one involving the $L^{2}$-orthogonal projection
$\nabla\Delta^{-1}\nabla$. For this, observe that
\[
\int_{D}|\nabla\Delta^{-1}\nabla\cdot\mathbf{v}|^{2}\,d\mathbf{x}=\max_{\theta\in H_{0}^{1}(D)}\,\int_{D}2\nabla\theta\cdot\mathbf{v}-|\nabla\theta|^{2}\,d\mathbf{x}
\]
for all $\mathbf{v}\in L^{2}(D)$. So, the estimate
\begin{equation}
\left\langle |\nabla\Delta^{-1}\nabla\cdot(\mathbf{u}\xi_{\delta})|^{2}\right\rangle \lesssim\delta^{2}\left\langle |\nabla\mathbf{u}|^{2}\right\rangle \label{eq:no-penetration-bound}
\end{equation}
follows from its steady version
\[
\int_{D}\nabla\theta\cdot\mathbf{u}\xi_{\delta}\lesssim\int_{D}|\nabla\theta|^{2}+\delta^{2}|\nabla\mathbf{u}|^{2}\quad\forall\,\theta\in H_{0}^{1}(D)
\]
which we now establish at a.e.\ time. Since $\mathbf{u}$ is divergence-free
and $\xi_{\delta}=\chi_{\delta}(r)$ depends only on $r$, 
\begin{align*}
\left|\int_{D}\nabla\theta\cdot\mathbf{u}\xi_{\delta}\right| & =\left|\int_{D}\theta\mathbf{u}\cdot\nabla\xi_{\delta}\right|\lesssim\int_{1-\delta}^{1}\left|\overline{\theta\mathbf{u}\cdot\hat{\mathbf{e}}_{r}}\right||\chi'_{\delta}|\,rdr\\
 & \lesssim\fint_{1-\delta}^{1}\left|\overline{\theta\mathbf{u}\cdot\hat{\mathbf{e}}_{r}}\right|\,dr.
\end{align*}
Note $\overline{\cdot}=\fint_{0}^{2\pi}\cdot\,d\theta$ gives the
$\theta$-average, as defined in \prettyref{subsec:Notation}. Continuing,
we claim that 
\begin{equation}
|\overline{\theta\mathbf{u}\cdot\hat{\mathbf{e}}_{r}}|\lesssim(1-r)||\nabla\theta||_{L^{2}(D)}||\nabla\mathbf{u}||_{L^{2}(D)}\quad\text{for a.e. }r\in(0,1).\label{eq:standard-estimate}
\end{equation}
In fact, \prettyref{eq:standard-estimate} follows by a more or less
standard argument involving the Cauchy\textendash Schwarz inequality
applied to $\frac{d}{dr}\overline{\theta^{2}}$ and $\frac{d}{dr}\overline{(\mathbf{u}\cdot\hat{\mathbf{e}}_{r})^{2}}$,
and then again to $\frac{d}{dr}\overline{\theta\mathbf{u}\cdot\hat{\mathbf{e}}_{r}}$,
all of which are derivatives of quantities vanishing at $r=1$. (For
the complete details in an analogous fluid layer setting, see \cite[Lemma 2.6]{tobasco2019optimal}.)
Applying \prettyref{eq:standard-estimate}, we get that
\begin{align*}
\fint_{1-\delta}^{1}\left|\overline{\theta\mathbf{u}\cdot\hat{\mathbf{e}}_{r}}\right|\,dr & \lesssim\fint_{1-\delta}^{1}(1-r)\,dr\cdot||\nabla\theta||_{L^{2}(D)}||\nabla\mathbf{u}||_{L^{2}(D)}\lesssim\delta||\nabla\theta||_{L^{2}(D)}||\nabla\mathbf{u}||_{L^{2}(D)}\\
 & \lesssim\int_{D}|\nabla\theta|^{2}+\delta^{2}|\nabla\mathbf{u}|^{2}.
\end{align*}
As $\theta\in H_{0}^{1}(D)$ was arbitrary, the estimate \prettyref{eq:no-penetration-bound}
is proved.

Summing up, we have shown via \prettyref{eq:bound-to-be-used}-\prettyref{eq:no-penetration-bound} with
the test function $\xi_{\delta}=\chi_{\delta}(r)$ that
\[
\left\langle f\right\rangle ^{2}\lesssim\left\langle |\nabla T|^{2}\right\rangle \left(\frac{1}{\delta}+\delta^{2}\left\langle |\nabla\mathbf{u}|^{2}\right\rangle \right)\quad\forall\,\delta\in(0,\delta_{0}(f)].
\]
The bound
\[
\left\langle f\right\rangle ^{2}\lesssim\left\langle |\nabla T|^{2}\right\rangle \left(\frac{1}{\delta_{0}}\vee\left\langle |\nabla\mathbf{u}|^{2}\right\rangle ^{1/3}\right)
\]
follows. Indeed, if $\delta_{0}^{-1}\leq\left\langle |\nabla\mathbf{u}|^{2}\right\rangle ^{1/3}$
we can choose $\delta=\left\langle |\nabla\mathbf{u}|^{2}\right\rangle ^{-1/3}$, otherwise we take $\delta=\delta_{0}$. This proves the
lower bound in \prettyref{thm:main-theorem} with $C\gtrsim\left\langle f\right\rangle ^{2}$
and $c=\delta_{0}^{-3}$. \end{proof}

A minor modification of the previous proof leads to the (supposedly sub-optimal) lower bound on energy-constrained cooling from the introduction. To achieve it, replace the inequality in \prettyref{eq:no-penetration-bound} with the simpler inequality
\begin{equation}
\left\langle |\nabla\Delta^{-1}\nabla\cdot(\mathbf{u}\xi_{\delta})|^{2}\right\rangle \leq \left\langle |\mathbf{u}|^{2}\right\rangle \label{eq:easy-upper-bound}
\end{equation}
which holds because $\nabla\Delta^{-1}\nabla \cdot$ is an $L^2$-orthogonal projection and $|\xi_\delta| \leq 1$. 
The rest of the proof goes through to yield the following result:
\begin{prop}
\label{prop:lower-bound-energy-constraint}
Let $f(\mathbf{x},t)$ satisfy
\[
\lim_{\tau\to\infty}\,\frac{1}{\sqrt{\tau}}\int_{0}^{\tau}e^{-\lambda_{1}\left((\tau-t)\wedge t\right)}||f(\cdot,t)||_{L^{2}(D)}\,dt=0\quad\text{and}\quad\left\langle |f|^2\right\rangle <\infty.
\]
There exist positive constants $C$ and $c(f)$, the former being numerical and the latter
depending only on $f$, such that
\[
\frac{C}{\emph{Pe}^{2}}\cdot\left\langle f\right\rangle ^{2}\leq\min_{\substack{\mathbf{u}(\mathbf{x},t)\\
\left\langle |\mathbf{u}|^{2}\right\rangle =\emph{Pe}^{2}
}
}\,\left\langle |\nabla T|^{2}\right\rangle 
\]
whenever $\emph{Pe}\geq c(f)$.
\end{prop}

The reader familiar with the background method may wonder whether
it also leads to these lower bounds. 
Indeed it does, the key point being that either \prettyref{eq:no-penetration-bound} or \prettyref{eq:easy-upper-bound}  
verifies the relevant spectral constraint, depending on whether one considers enstrophy- or energy-constrained cooling.
For more on the connection between
our symmetrization-based bounds and those of the background method,
see \cite{tobasco2019optimal,souza2020wall}.

\section{Optimal steady flows for enstrophy-constrained cooling\label{sec:Steady-branching-flows}}

The rest of this paper is about upper bounds. In this section and
the next, we prove the one from \prettyref{thm:main-theorem} on enstrophy-constrained
cooling. First, to warm-up, we study the steady version of the problem:
\begin{equation}
\min_{\substack{\substack{\mathbf{u}(\mathbf{x})\\
\fint_{D}|\nabla\mathbf{u}|^{2}=\text{Pe}^{2}
}
\\
\mathbf{u}=\mathbf{0}\text{ at }\partial D
}
}\,\left\langle |\nabla T|^{2}\right\rangle \label{eq:steady-pblm}
\end{equation}
for a given source $f(\mathbf{x})$. We return to the
unsteady problem in \prettyref{sec:Unsteady-branching-flows}. Note
we use no-slip boundary conditions from now on. This is
compatible with our goal of proving an upper bound on optimal no-penetration
flows, as minimizing over no-slip flows can only increase the result. 

When $\mathbf{u}$ and $f$ are steady, the upper bound from \prettyref{prop:aprioribds}
leads with a little effort to the double minimization 
\begin{equation}
\min_{\substack{\mathbf{u}(\mathbf{x}),\eta(\mathbf{x})\\
\mathbf{u}=\mathbf{0},\eta=0\text{ at }\partial D
}
}\,\fint_{D}|\nabla\Delta^{-1}\left(\mathbf{u}\cdot\nabla\eta-f\right)|^{2}+\frac{1}{\text{Pe}^{2}}\fint_{D}|\nabla\eta|^{2}\fint_{D}|\nabla\mathbf{u}|^{2}\label{eq:steady-double-min}
\end{equation}
whose optimal value is the same as that of the original problem \prettyref{eq:steady-pblm},
and whose optimizers $(\mathbf{u},\eta)$ yield solutions to it under
the rescaling $\mathbf{u}\to\lambda_{\text{Pe}}\mathbf{u}$ with $\lambda_{\text{Pe}}=\text{Pe}/\sqrt{\left\langle |\nabla\mathbf{u}|^{2}\right\rangle }$.
After explaining this in \prettyref{subsec:steady-change-of-vars},
we go on in \prettyref{subsec:Approximate-solutions-steady-advection}
to show how it is possible to achieve
\[
\mathbf{u}\cdot\nabla\eta\approx f\quad\text{in }H^{-1}(D)
\]
using convection roll-based flows. The estimates given there on the
non-local part of \prettyref{eq:steady-double-min} will come in handy later on 
when we discuss branching flows. 

\subsection{A change of variables\label{subsec:steady-change-of-vars}}

First, we show that the problems in \prettyref{eq:steady-pblm} and
\prettyref{eq:steady-double-min} are the same. Applying the upper
bound from \prettyref{prop:aprioribds}\textemdash here we use \prettyref{rem:steady-bds}
as it gives the better result in the steady case\textemdash we see
that
\[
\left\langle |\nabla T|^{2}\right\rangle =\min_{\substack{\eta(\mathbf{x})\\
\eta=0\text{ at }\partial D
}
}\,F(\mathbf{u},\eta),\quad\text{where}\quad F(\mathbf{u},\eta)=\fint_{D}|\nabla\Delta^{-1}\left(\mathbf{u}\cdot\nabla\eta-f\right)|^{2}+\fint_{D}|\nabla\eta|^{2}.
\]
Evidently, minimizing $\left\langle |\nabla T|^{2}\right\rangle $
is the same as minimizing $F$. A simple change of variables removes
the enstrophy constraint. Consider the substitutions
\begin{equation}
\mathbf{u}=\frac{\text{Pe}}{\sqrt{\fint_{D}|\nabla\tilde{\mathbf{u}}|^{2}}}\tilde{\mathbf{u}}\quad\text{and}\quad\eta=\frac{\sqrt{\fint_{D}|\nabla\tilde{\mathbf{u}}|^{2}}}{\text{Pe}}\tilde{\eta}\label{eq:change-of-vars}
\end{equation}
where $\tilde{\mathbf{u}}$ is not allowed to be identically zero.
Evidently, 
\[
\fint_{D}|\nabla\mathbf{u}|^{2}=\text{Pe}^{2}\quad\forall\,\tilde{\mathbf{u}}
\]
while
\[
F(\mathbf{u},\eta)=\fint_{D}|\nabla\Delta^{-1}(\tilde{\mathbf{u}}\cdot\nabla\tilde{\eta}-f)|^{2}+\frac{1}{\text{Pe}^{2}}\fint_{D}|\nabla\tilde{\mathbf{u}}|^{2}\cdot\fint_{D}|\nabla\tilde{\eta}|^{2}.
\]
Thus, the two minimization problems
\[
\min_{\substack{\substack{\mathbf{u}(\mathbf{x})\\
\fint_{D}|\nabla\mathbf{u}|^{2}=\text{Pe}^{2}
}
\\
\mathbf{u}=\mathbf{0}\text{ at }\partial D
}
}\,\left\langle |\nabla T|^{2}\right\rangle \quad\text{and}\quad\min_{\substack{\tilde{\mathbf{u}}(\mathbf{x}),\tilde{\eta}(\mathbf{x})\\
\tilde{\mathbf{u}}=\mathbf{0},\tilde{\eta}=0\text{ at }\partial D
}
}\,\fint_{D}|\nabla\Delta^{-1}(\tilde{\mathbf{u}}\cdot\nabla\tilde{\eta}-f)|^{2}+\frac{1}{\text{Pe}^{2}}\fint_{D}|\nabla\tilde{\mathbf{u}}|^{2}\cdot\fint_{D}|\nabla\tilde{\eta}|^{2}
\]
are the same. Their optimal values agree, and their optimizers
are related via \prettyref{eq:change-of-vars}. 

A similar change of variables can be done for the unsteady problem,
with the result being a bound rather than an equivalence. Setting
\[
\mathbf{u}=\frac{\text{Pe}}{\sqrt{\left\langle |\nabla\tilde{\mathbf{u}}|^{2}\right\rangle }}\tilde{\mathbf{u}}\quad\text{and}\quad\eta=\frac{\sqrt{\left\langle |\nabla\tilde{\mathbf{u}}|^{2}\right\rangle }}{\text{Pe}}\tilde{\eta}
\]
into the upper bound from \prettyref{prop:aprioribds} shows that
\[
\left\langle |\nabla T|^{2}\right\rangle \leq\min_{\substack{\tilde{\eta}(\mathbf{x},t)\\
\tilde{\eta}=0\text{ at }\partial D
}
}\,\left\langle \left|\nabla\Delta^{-1}\left[(\partial_{t}+\tilde{\mathbf{u}}\cdot\nabla)\tilde{\eta}-f\right]\right|^{2}+\frac{\left\langle |\nabla\tilde{\mathbf{u}}|^{2}\right\rangle }{\text{Pe}^{2}}|\nabla\tilde{\eta}|^{2}\right\rangle 
\]
where $T(\mathbf{x},t)$ is a temperature field associated to the
unsteady velocity $\mathbf{u}(\mathbf{x},t)$. This is the starting
point of \prettyref{sec:Unsteady-branching-flows}. 

\subsection{Steady advection\label{subsec:Approximate-solutions-steady-advection}}

The next step is to see what it takes to drive the non-local terms in these optimizations to zero. Focusing again on 
a steady source $f(\mathbf{x})$, we ask what it takes for a velocity\textendash test
function pair $(\mathbf{u}(\mathbf{x}),\eta(\mathbf{x}))$ to achieve
\[
\mathbf{u}\cdot\nabla\eta\approx f\quad\text{in }H^{-1}(D)
\]
where the notation means that $\fint_{D}|\nabla\Delta^{-1}\left(\mathbf{u}\cdot\nabla\eta-f\right)|^{2}$
is small. Guided by the divergence theorem and our usual assumption
that $\nabla\cdot\mathbf{u}=0$, we see that any successful pair must
 achieve
\[
\int_{\partial U}\mathbf{u}\eta\cdot\hat{\mathbf{n}}\,ds\approx\int_{U}f\,d\mathbf{x}
\]
for $U\subset D$. The following result makes this intuition precise. 

Introduce the notation $D_{r}$ for the disc of radius $r>0$ centered
at the origin of $D$, and again let
\[
\overline{\varphi}(r)=\frac{1}{2\pi}\int_{0}^{2\pi}\varphi(r,\theta)\,d\theta
\]
denote the average over its boundary $\partial D_{r}$. As noted in
\prettyref{subsec:Notation}, we allow ourselves to conflate
a point $\mathbf{x}$ with its polar coordinates $(r,\theta)$. Any
integral over $\rho$ is done with respect to the radial coordinate.
\begin{lem}
\label{lem:decomposition-quadratic-form}Let $f\in L^{2}(D)$ and
suppose that $(\mathbf{u},\eta)\in H^{1}(D;\mathbb{R}^{2})\times H^{1}(D)$
where $\mathbf{u}$ is divergence-free. Then, 
\[
\fint_{D}|\nabla\Delta^{-1}(\mathbf{u}\cdot\nabla\eta-f)|^{2}\,d\mathbf{x}=\frac{1}{2\pi}\int_{0}^{1}|\overline{\mathbf{u}\eta\cdot\hat{\mathbf{e}}_{r}}-F|^{2}\,rdr+Q(\mathbf{u}\eta-g\hat{\mathbf{e}}_{r})
\]
where
\begin{align*}
F(r) & =\frac{1}{2\pi r}\int_{D_{r}}f(\mathbf{x})\,d\mathbf{x},\quad g(r,\theta)=\frac{1}{r}\int_{0}^{r}\rho f(\rho,\theta)\,d\rho,\\
Q(\mathbf{v}) & =\min_{\varphi\in H^{1}(D)}\,\fint_{D}|-\frac{1}{r}\partial_{\theta}\varphi+\mathbf{v}\cdot\hat{\mathbf{e}}_{r}-\overline{\mathbf{v}\cdot\hat{\mathbf{e}}_{r}}|^{2}+|\partial_{r}\varphi+\mathbf{v}\cdot\hat{\mathbf{e}}_{\theta}|^{2}\,d\mathbf{x}.
\end{align*}
\end{lem}
\begin{rem}
\label{rem:good-test-function}A particular choice of test function
we use often below is
\[
\varphi=r\partial_{\theta}^{-1}(\mathbf{v}\cdot\hat{\mathbf{e}}_{r})=r\int_{0}^{\theta}\mathbf{v}\cdot\hat{\mathbf{e}}_{r}(r,\theta')-\overline{\mathbf{v}\cdot\hat{\mathbf{e}}_{r}}(r)\,d\theta'.
\]
It sets the first integral in $Q$ to zero, giving the bound
\[
Q(\mathbf{v})\leq\fint_{D}\left|\partial_{r}\left(r\partial_{\theta}^{-1}(\mathbf{v}\cdot\hat{\mathbf{e}}_{r})\right)+\mathbf{v}\cdot\hat{\mathbf{e}}_{\theta}\right|^{2}\,d\mathbf{x}.
\]
\end{rem}
\begin{proof}
We need the fact that
\begin{equation}
\int_{D}|\nabla\Delta^{-1}\nabla\cdot\mathbf{v}|^{2}=\min_{\substack{\mathbf{m}\in L^{2}(D)\\
\nabla\cdot\mathbf{m}=0
}
}\,\int_{D}|\mathbf{m}+\mathbf{v}|^{2}=\min_{\varphi\in H^{1}(D)}\,\int_{D}|\nabla^{\perp}\varphi+\mathbf{v}|^{2}\label{eq:elementary-fact}
\end{equation}
for any $\mathbf{v}\in L^{2}(D)$. A quick proof of it goes as follows:
let $\zeta\in H_{0}^{1}(D)$ satisfy $\Delta\zeta=\nabla\cdot\mathbf{v}$
in $D$, and note that $\nabla\zeta$ is $L^{2}$-orthogonal to divergence-free
$\mathbf{m}$, including $\nabla\zeta-\mathbf{v}$. Hence,
\[
\int_{D}|\mathbf{m}+\mathbf{v}|^{2}=\int_{D}|\mathbf{m}+\mathbf{v}-\nabla\zeta|^{2}+\int_{D}|\nabla\zeta|^{2}\geq\int_{D}|\nabla\zeta|^{2}=\int_{D}|\nabla\Delta^{-1}\nabla\cdot\mathbf{v}|^{2}
\]
and equality holds for $\mathbf{m}=\nabla\zeta-\mathbf{v}$. This
proves the first part of \prettyref{eq:elementary-fact}, and the
rest of it follows from the usual representation of a divergence-free
vector field $\mathbf{m}$ as the perpendicular gradient of a streamfunction
$\varphi$. 

Now using \prettyref{eq:elementary-fact} and the definition of $g$,
which satisfies $\partial_{r}(rg)=rf$, write that
\[
\int_{D}|\nabla\Delta^{-1}\left(\mathbf{u}\cdot\nabla\eta-f\right)|^{2}=\int_{D}|\nabla\Delta^{-1}\nabla\cdot\left(\mathbf{u}\eta-g\hat{\mathbf{e}}_{r}\right)|^{2}=\min_{\nabla\cdot\mathbf{m}=0}\,\int_{D}|\mathbf{m}+\mathbf{v}|^{2}
\]
where
\[
\mathbf{v}=\mathbf{u}\eta-g\hat{\mathbf{e}}_{r}.
\]
In the first step we used that $\nabla\cdot(g\hat{\mathbf{e}}_{r})=f$,
which is clear from its expression in polar coordinates. Observe that
a vector field $a(r)\hat{\mathbf{e}}_{r}$ is $L^{2}$-orthogonal
to any divergence-free $\mathbf{m}$. Indeed, writing
\[
\mathbf{m}=\nabla^{\perp}\varphi=-\frac{1}{r}\partial_{\theta}\varphi\hat{\mathbf{e}}_{r}+\partial_{r}\varphi\hat{\mathbf{e}}_{\theta}
\]
where $\varphi$ is $2\pi$-periodic in $\theta$, 
\[
\int_{D}\mathbf{m}\cdot a(r)\hat{\mathbf{e}}_{r}=\int_{0}^{1}\left[\int_{0}^{2\pi}-\frac{1}{r}\partial_{\theta}\varphi\,d\theta\right]a(r)\,rdr=0.
\]
This prompts the $L^{2}$-orthogonal decomposition
\[
\mathbf{v}=\overline{\mathbf{v}\cdot\hat{\mathbf{e}}_{r}}\hat{\mathbf{e}}_{r}+\mathbf{w}
\]
where $\mathbf{w}$ is the remainder. By orthogonality, 
\begin{align*}
\min_{\nabla\cdot\mathbf{m}=0}\,\fint_{D}|\mathbf{m}+\mathbf{v}|^{2} & =\min_{\nabla\cdot\mathbf{m}=0}\,\fint_{D}|\mathbf{m}+\mathbf{w}|^{2}+\fint_{D}|\overline{\mathbf{v}\cdot\hat{\mathbf{e}}_{r}}\hat{\mathbf{e}}_{r}|^{2}\\
 & =Q(\mathbf{u}\eta-g\hat{\mathbf{e}}_{r})+\int_{0}^{1}|\overline{\mathbf{u}\eta\cdot\hat{\mathbf{e}}_{r}}-\overline{g}|^{2}\,rdr
\end{align*}
as $\overline{\cdot}$ yields functions of $r$ alone. Since
\[
\overline{g}(r)=\frac{1}{2\pi}\int_{0}^{2\pi}\left[\frac{1}{r}\int_{0}^{r}\rho f\,d\rho\right]\,d\theta=\frac{1}{2\pi r}\int_{D_{r}}f=F(r)
\]
the result is proved. 
\end{proof}

We proceed to construct pairs $(\mathbf{u},\eta)$ which are not necessarily
admissible for the minimization in \prettyref{eq:steady-double-min},
but nevertheless do a good job at achieving $\mathbf{u}\cdot\nabla\eta\approx f$
in $H^{-1}(D)$ for a given $f\in L^{2}(D)$. We continue to use the
functions $F=\frac{1}{2\pi r}\int_{D_{r}}f$ and $g=\frac{1}{r}\int_{0}^{r}\rho f$
from \prettyref{lem:decomposition-quadratic-form}. Given $n\in\mathbb{N}$,
define the streamfunction 
\begin{equation}
\psi(\mathbf{x})=rg(r,\theta)\Psi(\theta)\quad\text{where}\quad\Psi(\theta)=\frac{\sqrt{2}}{n}\cos(n\theta)\label{eq:streamfunctions}
\end{equation}
whose velocity is
\begin{equation}
\mathbf{u}=\nabla^{\perp}\psi=-\left(\partial_{\theta}g\Psi+g\Psi'\right)\hat{\mathbf{e}}_{r}+rf\Psi\hat{\mathbf{e}}_{\theta}.\label{eq:velocity-warmup}
\end{equation}
Likewise, define the test function
\begin{equation}
\eta(\mathbf{x})=-\Psi'(\theta).\label{eq:testfunction-warmup}
\end{equation}

\begin{lem}
\label{lem:building-block} Let $f\in L^{2}(D)$. The velocity\textendash test
function pair $(\mathbf{u},\eta)$ in \prettyref{eq:streamfunctions}-\prettyref{eq:testfunction-warmup}
satisfies
\begin{equation}
\int_{D}|\nabla\Delta^{-1}(\mathbf{u}\cdot\nabla\eta-f)|^{2}\,d\mathbf{x}\lesssim\frac{1}{n^{2}}\int_{D}|f|^{2}+|\nabla f|^{2}\,d\mathbf{x}.\label{eq:advection-eqn-bound}
\end{equation}
In particular, we have the estimates
\begin{align}
\left|\overline{\mathbf{u}\eta\cdot\hat{\mathbf{e}}_{r}}-F\right| & \lesssim\frac{1}{n}\int_{D_{r}}|\nabla f|\,d\mathbf{x},\label{eq:flux-bound}\\
||\partial_{r}(r\partial_{\theta}^{-1}(\mathbf{u}\eta\cdot\hat{\mathbf{e}}_{r}-g))||_{L_{\theta}^{1}} & \lesssim\frac{1}{n}\left(\int_{D_{r}}\frac{|f|}{r}+|\nabla f|\,d\mathbf{x}+r||f||_{L_{\theta}^{1}}+r^{2}||\nabla f||_{L_{\theta}^{1}}\right),\label{eq:flux-bound-non-local}\\
||\mathbf{u}\eta\cdot\hat{\mathbf{e}}_{\theta}||_{L_{\theta}^{1}} & \lesssim\frac{r}{n}||f||_{L_{\theta}^{1}}\label{eq:flux-bound-local}
\end{align}
for a.e.\ $r>0$. Furthermore, the streamfunction $\psi$ and velocity
$\mathbf{u}$ obey 
\begin{align}
|\psi| & \lesssim\frac{1}{n}\int_{0}^{r}\rho|f|\,d\rho,\quad|\mathbf{u}\cdot\hat{\mathbf{e}}_{r}|\lesssim\int_{0}^{r}\frac{\rho|f|}{r}+\frac{\rho|\nabla f|}{n}\,d\rho,\quad|\mathbf{u}\cdot\hat{\mathbf{e}}_{\theta}|\lesssim\frac{r}{n}|f|,\label{eq:velocity-bd}\\
|\nabla\mathbf{u}| & \lesssim\int_{0}^{r}\frac{n\rho|f|}{r^{2}}+\frac{n\rho|\nabla f|}{r}+\frac{\rho|\nabla\nabla f|}{n}\,d\rho+|f|+r|\nabla f|\label{eq:enstrophy-bd}
\end{align}
and the test function $\eta$ obeys
\begin{equation}
|\eta|\lesssim1\quad\text{and}\quad|\nabla\eta|\lesssim\frac{n}{r}\label{eq:eta-estimates}
\end{equation}
for a.e.\ $\theta\in[0,2\pi]$ and $r>0$. The constants implicit
in these estimates are independent of all parameters. 
\end{lem}
\begin{proof}
We start at the bottom of the claim and work backwards. The estimates
in \prettyref{eq:eta-estimates} are clear given the formula for $\eta$.
To prove \prettyref{eq:velocity-bd} and \prettyref{eq:enstrophy-bd},
we require the following inequalities involving $g=\frac{1}{r}\int_{0}^{r}\rho f$:
\begin{align}
|g| & \leq\frac{1}{r}\int_{0}^{r}\rho|f|\,d\rho,\quad|\partial_{r}g|\leq\frac{1}{r^{2}}\int_{0}^{r}\rho|f|\,d\rho+|f|,\quad|\partial_{\theta}g|\leq\int_{0}^{r}\rho|\nabla f|\,d\rho\label{eq:estimates-on-g}\\
|\partial_{r\theta}g| & \leq\frac{1}{r}\int_{0}^{r}\rho|\nabla f|\,d\rho+r|\nabla f|,\quad|\partial_{\theta\theta}g|\leq r\int_{0}^{r}\rho|\nabla\nabla f|\,d\rho+\int_{0}^{r}\rho|\nabla f|\,d\rho.\label{eq:estimates-on-gderivs}
\end{align}
Now by the definitions of $\psi$ and $\mathbf{u}$,
\[
|\psi|\lesssim\frac{r}{n}|g|,\quad|\mathbf{u}\cdot\hat{\mathbf{e}}_{r}|\lesssim\frac{1}{n}|\partial_{\theta}g|+|g|,\quad|\mathbf{u}\cdot\hat{\mathbf{e}}_{\theta}|\lesssim\frac{r}{n}|f|
\]
and \prettyref{eq:velocity-bd} is proved. Next, compute the gradient
\begin{align*}
\nabla\mathbf{u} & =-\nabla\hat{\mathbf{e}}_{r}\left(\partial_{\theta}g\Psi+g\Psi'\right)-\hat{\mathbf{e}}_{r}\otimes\nabla\left(\partial_{\theta}g\Psi+g\Psi'\right)\\
 & \qquad+\nabla\hat{\mathbf{e}}_{\theta}\left(rf\Psi\right)+\hat{\mathbf{e}}_{\theta}\otimes\nabla\left(rf\Psi\right)\\
 & =-\hat{\mathbf{e}}_{\theta}\otimes\hat{\mathbf{e}}_{\theta}\frac{1}{r}(\partial_{\theta}g\Psi+g\Psi')-\hat{\mathbf{e}}_{r}\otimes\nabla\left(\partial_{\theta}g\Psi+g\Psi'\right)\\
 & \qquad-\hat{\mathbf{e}}_{r}\otimes\hat{\mathbf{e}}_{\theta}f\Psi+\hat{\mathbf{e}}_{\theta}\otimes\nabla\left(rf\Psi\right)
\end{align*}
where
\begin{align*}
\nabla\left(\partial_{\theta}g\Psi+g\Psi'\right) & =\hat{\mathbf{e}}_{r}\left(\partial_{r\theta}g\Psi+\partial_{r}g\Psi'\right)+\hat{\mathbf{e}}_{\theta}\frac{1}{r}\left(\partial_{\theta\theta}g\Psi+2\partial_{\theta}g\Psi'+g\Psi''\right),\\
\nabla\left(rf\Psi\right) & =\hat{\mathbf{e}}_{r}\left(f+r\partial_{r}f\right)\Psi+\hat{\mathbf{e}}_{\theta}\left(\partial_{\theta}f\Psi+f\Psi'\right).
\end{align*}
Hence,
\begin{align*}
|\nabla\mathbf{u}| & \leq\frac{1}{r}\left|\partial_{\theta}g\Psi+g\Psi')\right|+\left|\partial_{r\theta}g\Psi+\partial_{r}g\Psi'\right|+\frac{1}{r}\left|\partial_{\theta\theta}g\Psi+2\partial_{\theta}g\Psi'+g\Psi''\right|\\
 & \qquad+\left|f\Psi\right|+\left|\left(f+r\partial_{r}f\right)\Psi\right|+\left|\partial_{\theta}f\Psi+f\Psi'\right|\\
 & \lesssim\left(|g|+|\partial_{\theta}g|\right)\frac{n}{r}+|f|+|\partial_{r}g|\\
 & \qquad+\left(|\partial_{r}f|+\frac{1}{r}|\partial_{\theta}f|+\frac{1}{r}|\partial_{r\theta}g|+\frac{1}{r^{2}}|\partial_{\theta\theta}g|\right)\frac{r}{n}.
\end{align*}
Combined with \prettyref{eq:estimates-on-g} and \prettyref{eq:estimates-on-gderivs},
this yields the desired estimates in \prettyref{eq:enstrophy-bd}. 

It remains to show the first half of the claim. We require the inequalities
\begin{align}
||g||_{L_{\theta}^{1}} & \leq\frac{1}{r}\int_{D_{r}}|f|\,d\mathbf{x},\quad||\partial_{r}g||_{L_{\theta}^{1}}\leq\frac{1}{r^{2}}\int_{D_{r}}|f|\,d\mathbf{x}+||f||_{L_{\theta}^{1}},\label{eq:integrated-estimates-on-g}\\
||\partial_{\theta}g||_{L_{\theta}^{1}} & \leq\int_{D_{r}}|\nabla f|\,d\mathbf{x},\quad||\partial_{r\theta}g||_{L_{\theta}^{1}}\leq\frac{1}{r}\int_{D_{r}}|\nabla f|\,d\mathbf{x}+r||\nabla f||_{L_{\theta}^{1}}\label{eq:integrated-estimates-on-g-derivs}
\end{align}
which follow from \prettyref{eq:estimates-on-g} and \prettyref{eq:estimates-on-gderivs}
by integration. Going back to the definitions, 
\begin{align*}
\mathbf{u}\eta-g\hat{\mathbf{e}}_{r} & =\left(\partial_{\theta}g\Psi+g\Psi'\right)\Psi'\hat{\mathbf{e}}_{r}-rf\Psi\Psi'\hat{\mathbf{e}}_{\theta}-g\hat{\mathbf{e}}_{r}\\
 & =\left(\partial_{\theta}g\Psi\Psi'+g\left((\Psi')^{2}-1\right)\right)\hat{\mathbf{e}}_{r}-rf\Psi\Psi'\hat{\mathbf{e}}_{\theta}.
\end{align*}
Averaging the $\hat{\mathbf{e}}_{r}$-component in $\theta$ and using that
$F=\overline{g}$, there follows
\begin{align*}
\left|\overline{\mathbf{u}\eta\cdot\hat{\mathbf{e}}_{r}}-F\right| & =\left|\overline{\mathbf{u}\eta\cdot\hat{\mathbf{e}}_{r}-g}\right|\leq\left|\fint_{0}^{2\pi}\partial_{\theta}g\Psi\Psi'\right|+\left|\fint_{0}^{2\pi}g\left((\Psi')^{2}-1\right)\right|\\
 & \lesssim||\partial_{\theta}g||_{L_{\theta}^{1}}\frac{1}{n}+\left|\fint_{0}^{2\pi}g\left((\Psi')^{2}-1\right)\right|
\end{align*}
since $|\Psi|\lesssim1/n$ and $|\Psi'|\lesssim1$. Introduce the
operator $\partial_{\theta}^{-1}$ defined by 
\[
\partial_{\theta}^{-1}\varphi(r,\theta)=\int_{0}^{\theta}\varphi(r,\theta')-\overline{\varphi}(r)\,d\theta'
\]
and observe that $\overline{(\Psi')^{2}}=1$. Hence, 
\begin{align}
\left|\fint_{0}^{2\pi}g\left((\Psi')^{2}-1\right)\right| & =\left|\fint_{0}^{2\pi}\partial_{\theta}g\partial_{\theta}^{-1}(\Psi')^{2}\right|\nonumber \\
 & \lesssim||\partial_{\theta}g||_{L_{\theta}^{1}}||\partial_{\theta}^{-1}(\Psi')^{2}||_{L_{\theta}^{\infty}}\lesssim||\partial_{\theta}g||_{L_{\theta}^{1}}\frac{1}{n}\label{eq:Lebesgue-lemma-avg}
\end{align}
and \prettyref{eq:flux-bound} is proved.

Continuing with \prettyref{eq:flux-bound-non-local}, write that
\begin{align*}
|\partial_{r}(r\partial_{\theta}^{-1}(\mathbf{u}\eta\cdot\hat{\mathbf{e}}_{r}-g))| & \leq|\partial_{\theta}^{-1}(\mathbf{u}\eta\cdot\hat{\mathbf{e}}_{r}-g)|+|r\partial_{r}\partial_{\theta}^{-1}(\mathbf{u}\eta\cdot\hat{\mathbf{e}}_{r}-g)|\\
 & =|\partial_{\theta}^{-1}\left[\partial_{\theta}g\Psi\Psi'+g\left((\Psi')^{2}-1\right)\right]|+r|\partial_{r}\partial_{\theta}^{-1}\left[\partial_{\theta}g\Psi\Psi'+g\left((\Psi')^{2}-1\right)\right]|.
\end{align*}
Of course,
\[
|\partial_{\theta}^{-1}\left[\partial_{\theta}g\Psi\Psi'\right]|\lesssim||\partial_{\theta}g||_{L_{\theta}^{1}}\frac{1}{n}\quad\text{and}\quad|\partial_{r}\partial_{\theta}^{-1}\left[g\Psi\Psi'\right]|\lesssim||\partial_{r\theta}g||_{L_{\theta}^{1}}\frac{1}{n}.
\]
Using the operator $\partial_{\theta}^{-1}$ again, we see that
\begin{align*}
\left|\partial_{\theta}^{-1}\left[g\left((\Psi')^{2}-1\right)\right]\right| & \leq\left|\int_{0}^{\theta}g\left((\Psi')^{2}-1\right)\right|+\left|\int_{0}^{2\pi}g\left((\Psi')^{2}-1\right)\right|,\\
\left|\partial_{r}\partial_{\theta}^{-1}\left[g\left((\Psi')^{2}-1\right)\right]\right| & \leq\left|\int_{0}^{\theta}\partial_{r}g\left((\Psi')^{2}-1\right)\right|+\left|\int_{0}^{2\pi}\partial_{r}g\left((\Psi')^{2}-1\right)\right|.
\end{align*}
The last two terms on the righthand side are controlled by the inequality
\prettyref{eq:Lebesgue-lemma-avg} and one just like it with $\partial_{r}g$
in place of $g$. To deal with the first two terms on the righthand
side above, we require the inequalities
\begin{align}
\left\Vert \int_{0}^{\theta}g\left((\Psi')^{2}-1\right)\right\Vert _{L_{\theta}^{1}} & \lesssim\left(||g||_{L_{\theta}^{1}}+||\partial_{\theta}g||_{L_{\theta}^{1}}\right)\frac{1}{n},\label{eq:Lebesgue-lemma-integral1}\\
\left\Vert \int_{0}^{\theta}\partial_{r}g\left((\Psi')^{2}-1\right)\right\Vert _{L_{\theta}^{1}} & \lesssim\left(||\partial_{r}g||_{L_{\theta}^{1}}+||\partial_{r\theta}g||_{L_{\theta}^{1}}\right)\frac{1}{n}.\label{eq:Lebesgue-lemma-integral2}
\end{align}
To prove them, divide $[0,2\pi)$ into the disjoint intervals $I_{j}=[2\pi j/n,2\pi(j+1)/n)$
indexed by $j=0,\dots,n-1$, and note that $(\Psi')^{2}-1$ is $\frac{2\pi}{n}$-periodic
and averages to zero over each $I_{j}$. Choosing $k$ such that $\theta\in I_{k}$,
write that
\begin{align*}
\left|\int_{0}^{\theta}g\left((\Psi')^{2}-1\right)\right| & \leq\sum_{j=0}^{k-1}\left|\int_{I_{j}}\left(g-\fint_{I_{j}}g\right)\left((\Psi')^{2}-1\right)\right|+\left|\int_{2\pi k/n}^{\theta}g\left((\Psi')^{2}-1\right)\right|\\
 & \lesssim\sum_{j=0}^{k-1}\int_{I_{j}}\left|g-\fint_{I_{j}}g\right|+\int_{2\pi k/n}^{\theta}|g|\lesssim||\partial_{\theta}g||_{L_{\theta}^{1}}\frac{1}{n}+||g||_{L_{\theta}^{1}(I_{k})}.
\end{align*}
Remembering that $k=k(\theta)$ and integrating this bound over $\theta\in[0,2\pi)$
yields \prettyref{eq:Lebesgue-lemma-integral1}. The proof of the
inequality from \prettyref{eq:Lebesgue-lemma-integral2} is much the
same. Altogether, we have shown that
\begin{align}
\left\Vert \partial_{\theta}^{-1}\left[\partial_{\theta}g\Psi\Psi'+g\left((\Psi')^{2}-1\right)\right]\right\Vert _{L_{\theta}^{1}} & \lesssim\left(||g||_{L_{\theta}^{1}}+||\partial_{\theta}g||_{L_{\theta}^{1}}\right)\frac{1}{n},\label{eq:final-Lebesgue1}\\
\left\Vert \partial_{r}\partial_{\theta}^{-1}\left[\partial_{\theta}g\Psi\Psi'+g\left((\Psi')^{2}-1\right)\right]\right\Vert _{L_{\theta}^{1}} & \lesssim\left(||\partial_{r}g||_{L_{\theta}^{1}}+||\partial_{r\theta}g||_{L_{\theta}^{1}}\right)\frac{1}{n}\label{eq:final-Lebesgue2}
\end{align}
which when combined with the inequalities in \prettyref{eq:integrated-estimates-on-g}
and \prettyref{eq:integrated-estimates-on-g-derivs} lead to the estimate
\prettyref{eq:flux-bound-non-local} in the claim. The estimate \prettyref{eq:flux-bound-local}
follows from what we have already proved (namely, \prettyref{eq:velocity-bd}
and \prettyref{eq:eta-estimates}).

Finally, we prove \prettyref{eq:advection-eqn-bound}. By \prettyref{lem:decomposition-quadratic-form}
and the remark appearing immediately after, 
\[
\int_{D}|\nabla\Delta^{-1}(\mathbf{u}\cdot\nabla\eta-f)|^{2}\lesssim\int_{0}^{1}|\overline{\mathbf{u}\eta\cdot\hat{\mathbf{e}}_{r}}-F|^{2}\,rdr+\int_{D}|\partial_{r}\left(r\partial_{\theta}^{-1}(\mathbf{u}\eta\cdot\hat{\mathbf{e}}_{r}-g)\right)|^{2}+\int_{D}|\mathbf{u}\eta\cdot\hat{\mathbf{e}}_{\theta}|^{2}.
\]
Applying \prettyref{eq:flux-bound}-\prettyref{eq:flux-bound-local}
along with Jensen's inequality gives the bound
\[
\int_{0}^{1}|\overline{\mathbf{u}\eta\cdot\hat{\mathbf{e}}_{r}}-F|^{2}\,rdr\lesssim\int_{0}^{1}\left(\frac{r^{2}}{n^{2}}\int_{D_{r}}|\nabla f|^{2}\right)\,rdr\lesssim\frac{1}{n^{2}}\int_{D}|\nabla f|^{2}
\]
as well as the bounds
\begin{align*}
\int_{D}|\partial_{r}\left(r\partial_{\theta}^{-1}(\mathbf{u}\eta\cdot\hat{\mathbf{e}}_{r}-g)\right)|^{2} & \lesssim\int_{0}^{1}\frac{1}{n^{2}}\left(\int_{D_{r}}|f|^{2}+r^{2}\int_{D_{r}}|\nabla f|^{2}+r^{2}||f||_{L_{\theta}^{2}}^{2}+r^{4}||\nabla f||_{L_{\theta}^{2}}^{2}\right)\,rdr\\
 & \lesssim\frac{1}{n^{2}}\int_{D}|f|^{2}+|\nabla f|^{2}
\end{align*}
and
\[
\int_{D}|\mathbf{u}\eta\cdot\hat{\mathbf{e}}_{\theta}|^{2}\lesssim\int_{0}^{1}\left(\frac{r^{2}}{n^{2}}||f||_{L_{\theta}^{2}}^{2}\right)\,rdr\lesssim\frac{1}{n^{2}}\int_{D}|f|^{2}.
\]
The proof is complete.
\end{proof}
As a quick application of this last result, we show that the velocities
in \prettyref{eq:velocity-warmup} have finite enstrophy if $f$ is
sufficiently regular. By the pointwise estimate from \prettyref{eq:enstrophy-bd}
and Jensen's inequality, 
\begin{align*}
|\nabla\mathbf{u}|^{2} & \lesssim\int_{0}^{r}\left[\frac{n^{2}\rho^{2}}{r^{3}}|f|^{2}+\frac{n^{2}\rho^{2}}{r}|\nabla f|^{2}+\frac{\rho^{2}r}{n^{2}}|\nabla\nabla f|^{2}\right]\,d\rho\\
 & \qquad+|f|^{2}+r^{2}|\nabla f|^{2}.
\end{align*}
Hence,
\begin{align*}
\int_{D}|\nabla\mathbf{u}|^{2}\,d\mathbf{x} & \lesssim\int_{0}^{2\pi}\int_{0\leq\rho\leq r\leq1}\left[\frac{n^{2}\rho^{2}}{r^{2}}|f(\rho,\theta)|^{2}+n^{2}\rho^{2}|\nabla f(\rho,\theta)|^{2}+\frac{\rho^{2}r^{2}}{n^{2}}|\nabla\nabla f(\rho,\theta)|^{2}\right]\,drd\rho d\theta\\
 & \qquad+\int_{D}|f|^{2}+r^{2}|\nabla f|^{2}\,d\mathbf{x}\\
 & \lesssim\int_{D}n^{2}|f|^{2}+n^{2}r|\nabla f|^{2}+\frac{r}{n^{2}}|\nabla\nabla f|^{2}\,d\mathbf{x}.
\end{align*}

\section{Unsteady branching flows for enstrophy-constrained cooling\label{sec:Unsteady-branching-flows}}

\prettyref{sec:Steady-branching-flows} considered the steady optimal
cooling problem and explained how to find ``approximate $H^{-1}$-solutions''
to the corresponding advection equation $\mathbf{u}\cdot\nabla\eta=f$.
We now return to the original unsteady setting of \prettyref{thm:main-theorem},
to prove our upper bound on 
\[
\min_{\substack{\substack{\mathbf{u}(\mathbf{x},t)\\
\left\langle |\nabla\mathbf{u}|^{2}\right\rangle =\text{Pe}^{2}
}
\\
\mathbf{u}=\mathbf{0}\text{ at }\partial D
}
}\,\left\langle |\nabla T|^{2}\right\rangle 
\]
for a general source $f(\mathbf{x},t)$. We do so by constructing a family of well-chosen branching flows $\{\mathbf{u}_{\text{Pe}}\}$
whose temperatures $\{T_{\text{Pe}}\}$ satisfy 
\[
\left\langle |\nabla T_{\text{Pe}}|^{2}\right\rangle \lesssim\left\langle |f|^{2}+|\nabla f|^{2}+|\nabla\nabla f|^{2}\right\rangle \frac{\log^{4/3}\text{Pe}}{\text{Pe}^{2/3}}\quad\text{with}\quad\left\langle |\nabla\mathbf{u}_{\text{Pe}}|^{2}\right\rangle =\text{Pe}^{2}.
\]
\prettyref{subsec:Branching-flows} starts by defining a general family
of convection roll-based branching flows (see the bottom row of \prettyref{fig:flows}). 
\prettyref{subsec:branching-bounds}
estimates their cooling and \prettyref{subsec:Optimal-branching-flows}
optimizes over their free parameters. The upper bound from \prettyref{thm:main-theorem}
is finally proved at the end of this section.

Picking up where we left off in \prettyref{sec:Steady-branching-flows},
recall the upper bound
\[
\min_{\substack{\substack{\mathbf{u}(\mathbf{x},t)\\
\left\langle |\nabla\mathbf{u}|^{2}\right\rangle =\text{Pe}^{2}
}
\\
\mathbf{u}=\mathbf{0}\text{ at }\partial D
}
}\,\left\langle |\nabla T|^{2}\right\rangle \leq\min_{\substack{\mathbf{u}(\mathbf{x},t),\eta(\mathbf{x},t)\\
\mathbf{u}=\mathbf{0},\eta=0\text{ at }\partial D
}
}\,\left\langle |\nabla\Delta^{-1}\left[(\partial_{t}+\mathbf{u}\cdot\nabla)\eta-f\right]|^{2}+\frac{\left\langle |\nabla\mathbf{u}|^{2}\right\rangle }{\text{Pe}^{2}}|\nabla\eta|^{2}\right\rangle 
\]
where on the righthand side the magnitude of $\mathbf{u}$ is unconstrained.
This bound follows from \prettyref{prop:aprioribds} and the change
of variables $(\mathbf{u},\eta)\to(\lambda_{\text{Pe}}\mathbf{u},\lambda_{\text{Pe}}^{-1}\eta)$
with $\lambda_{\text{Pe}}=\text{Pe}/\sqrt{\left\langle |\nabla\mathbf{u}|^{2}\right\rangle }$,
as explained in \prettyref{subsec:steady-change-of-vars}. A special
case occurs for a steady test function $\eta(\mathbf{x})$: the temperature
field $T$ associated to $\lambda_{\text{Pe}}\mathbf{u}$ obeys
\[
\left\langle |\nabla T|^{2}\right\rangle \leq\left\langle |\nabla\Delta^{-1}\left[\mathbf{u}\cdot\nabla\eta-f\right]|^{2}\right\rangle +\frac{1}{\text{Pe}^{2}}\left\langle |\nabla\mathbf{u}|^{2}\right\rangle \fint_{D}|\nabla\eta|^{2}
\]
for all $\eta\in H_{0}^{1}(D)$. We proceed to define our branching
flows.

\subsection{Branching flows\label{subsec:Branching-flows}}

By a \emph{branching flow} $\mathbf{u}(\mathbf{x},t)$ and a corresponding
(steady) test function $\eta(\mathbf{x})$, we mean a divergence-free
velocity field
\[
\mathbf{u}=\nabla^{\perp}\psi\quad\text{where}\quad\psi(\mathbf{x},t)=\sum_{k=1}^{n}\chi_{k}(r)\psi_{k}(\mathbf{x},t)
\]
and the scalar function
\[
\eta(\mathbf{x})=\sum_{k=1}^{n}\chi_{k}(r)\eta_{k}(\mathbf{x})
\]
where $\{\chi_{k}\}$, $\{\psi_{k}\}$, and $\{\eta_{k}\}$ are as
follows. Let
\begin{equation}
F(r,t)=\frac{1}{2\pi r}\int_{D_{r}}f(\mathbf{x},t)\,d\mathbf{x}\quad\text{and}\quad g(r,\theta,t)=\fint_{0}^{r}\rho f(\rho,\theta,t)\,d\rho\label{eq:unsteady-Fandg}
\end{equation}
and set 
\[
\psi_{k}=rg\Psi_{k}(\theta),\quad\eta_{k}=-\Psi_{k}'(\theta),\quad\text{and}\quad\Psi_{k}(\theta)=\sqrt{2}l_{k}\cos(\frac{\theta}{l_{k}})\quad\text{for }k=1,\dots,n.
\]
The parameters $\{l_{k}^{-1}\}\subset\mathbb{N}$ and $n\in\mathbb{N}$
are free, and will eventually be optimized when it comes time to prove
\prettyref{thm:main-theorem}. The functions $\{\chi_{k}\}$ are described
in the paragraph after the next. 

To help organize the discussion, we always assume that
\begin{equation}
l_{1}>l_{2}>\cdots>l_{n}.\label{eq:lengths-decrease}
\end{equation}
We label the largest and smallest scales as
\[
l_{\text{bulk}}=l_{1}\quad\text{and}\quad l_{\text{bl}}=l_{n}
\]
noting that they occur in the bulk of the disc and in a boundary layer
near $r=1$, respectively. The individual velocities 
\[
\mathbf{u}_{k}=\nabla^{\perp}\psi_{k}=-\left(\partial_{\theta}g\Psi_{k}+g\Psi_{k}'\right)\hat{\mathbf{e}}_{r}+rf\Psi_{k}\hat{\mathbf{e}}_{\theta}
\]
are simply unsteady versions of the ones occurring in our prior discussion
of roll-like flows, and so are governed by the estimates in \prettyref{lem:building-block}.
The individual test functions 
\[
\eta_{k}=-\Psi_{k}'
\]
are also like those in the lemma. 

The new ingredients are the functions $\{\chi_{k}\}$ which we use
to interpolate between the individual building blocks listed above.
We use a family of smooth and compactly supported ``cutoff functions''
defined via a choice of points $\{r_{k}\}$ satisfying 
\begin{equation}
\frac{1}{2}<r_{1}<r_{2}<\cdots<r_{n}<1,\label{eq:points-increase}
\end{equation}
with
\[
r_{1}=r_{\text{bulk}}\quad\text{and}\quad r_{n}=r_{\text{bl}}.
\]
These functions can be quite general, but to fix ideas we let
\begin{equation}
\text{supp}\,\chi_{1}\subset(0,r_{2}),\quad\text{supp}\,\chi_{n}\subset(r_{n-1},1),\quad\text{and}\quad\text{supp}\,\chi_{k}\subset(r_{k-1},r_{k+1})\quad\text{for }k=2,\dots,n-1.\label{eq:bump-function-supports}
\end{equation}
We also let
\begin{equation}
\sum_{k=1}^{n}(\chi_{k}(r))^{2}=1\quad\forall\,r\in(0,r_{\text{bl}})\label{eq:Pythagorean-sum}
\end{equation}
and take
\begin{equation}
\chi_{k}\chi_{k'}\neq0\iff|k-k'|\leq1.\label{eq:nearly-diagonal}
\end{equation}
This last condition simplifies the calculation of products such as
$\mathbf{u}\eta$.

To specify our cutoff functions further, introduce the lengths 
\[
\delta_{\text{bulk}}=1-r_{\text{bulk}},\quad\delta_{\text{bl}}=1-r_{\text{bl}},\quad\text{and}\quad\delta_{k}=r_{k+1}-r_{k}\quad\text{for }k=1,\dots,n-1
\]
and, following the pattern, call
\[
\delta_{n}=\delta_{\text{bl}}.
\]
Let
\begin{equation}
|\chi_{k}|\vee|\chi_{k+1}|\lesssim1,\quad|\chi_{k}'|\vee|\chi_{k+1}'|\lesssim\frac{1}{\delta_{k}},\quad|\chi_{k}''|\vee|\chi_{k+1}''|\lesssim\frac{1}{\delta_{k}^{2}}\quad\forall\,r\in(r_{k},r_{k+1})\label{eq:bump-functions-general}
\end{equation}
for $k=1,\dots,n-1$, and let
\begin{equation}
|\chi_{n}|\lesssim1,\quad|\chi_{n}'|\lesssim\frac{1}{\delta_{\text{bl}}},\quad|\chi_{n}''|\lesssim\frac{1}{\delta_{\text{bl}}^{2}}\quad\forall\,r\in(r_{\text{bl}},1).\label{eq:bump-functions-boundarylayer}
\end{equation}
The constants implicit in these hypotheses are independent of all
parameters.

The reader looking for specific cutoff functions should consult \cite[Section 5.1]{tobasco2019optimal}.
There, we describe a similar branching flow in a fluid layer, the
$z$-coordinate of which is analogous to $r$. The exact choice of
these functions does not affect the scaling of our bounds in $\text{Pe}$, but does affect their prefactors.

\subsection{Upper bounds on branching flows\label{subsec:branching-bounds}}

Having defined our branching flows, we proceed to estimate their cooling
using the results of \prettyref{sec:Variational-bounds}. We write
the estimate in terms of a continuously varying version of the parameters
$\{l_{k}\}$ and $\{r_{k}\}$, given by
\begin{equation}
\ell(r)=l_{k}\frac{r_{k+1}-r}{r_{k+1}-r_{k}}+l_{k+1}\frac{r-r_{k}}{r_{k+1}-r_{k}}\quad\text{for }r\in[r_{k},r_{k+1}]\label{eq:continuous-scale}
\end{equation}
and $k=1,\dots,n-1$. By construction, $\ell(r_{k})=l_{k}$ for each
$k$. Recall $\delta_{k}=r_{k+1}-r_{k}$.
\begin{prop}
\label{prop:crucial-estimate} Let $\{\mathbf{u}\}$ be a family of
branching flows as defined in \prettyref{subsec:Branching-flows},
whose parameters $\{l_{k}\}_{k=1}^{n}$ and $\{r_{k}\}_{k=1}^{n}$
obey 
\begin{align}
|l_{k+1}-l_{k}| & \sim l_{k+1}\sim l_{k}\quad\text{and}\quad\delta_{k+1}\sim\delta_{k}\quad\text{for }k=1,\dots,n-1\label{eq:branching-analysis-assumptions-1}\\
l_{k}\lesssim\delta_{k} & \quad\text{for }k=1,\dots,n\label{eq:branching-analysis-assumptions-2}
\end{align}
with fixed numerical prefactors. Define the rescaled velocities 
\[
\lambda_{\emph{Pe}}\mathbf{u}\quad\text{with}\quad\lambda_{\emph{Pe}}=\frac{\emph{Pe}}{\sqrt{\left\langle |\nabla\mathbf{u}|^{2}\right\rangle }}
\]
and let their temperature fields $T_{\emph{Pe}}$ solve
\[
\begin{cases}
\partial_{t}T_{\emph{Pe}}+\lambda_{\emph{Pe}}\mathbf{u}\cdot\nabla T_{\emph{Pe}}=\Delta T_{\emph{Pe}}+f & \text{in }D\\
T_{\emph{Pe}}=0 & \text{at }\partial D
\end{cases}
\]
weakly with arbitrary $L^{2}$-initial data. The estimate 
\[
\left\langle |\nabla T_{\emph{Pe}}|^{2}\right\rangle \lesssim C_{0}(f)\left[l_{\emph{bulk}}^{2}+\int_{r_{\emph{bulk}}}^{r_{\emph{bl}}}(\ell'(r))^{2}\,dr+\delta_{\emph{bl}}+\frac{1}{\emph{Pe}^{2}}\left(\frac{1}{l_{\emph{bulk}}^{2}}+\int_{r_{\emph{bulk}}}^{r_{\emph{bl}}}\frac{1}{(\ell(r))^{2}}\,dr+\frac{\delta_{\emph{bl}}}{l_{\emph{bl}}^{2}}\right)^{2}\right]
\]
holds with
\[
C_{0}(f)=\left\langle |f|^{2}+|\nabla f|^{2}+|\nabla\nabla f|^{2}\right\rangle 
\]
and a numerical prefactor depending only on the ones from \prettyref{eq:branching-analysis-assumptions-1}
and \prettyref{eq:branching-analysis-assumptions-2}.
\end{prop}
By \prettyref{prop:aprioribds}, we already know that the rescaled
velocities $\lambda_{\text{Pe}}\mathbf{u}$ achieve
\[
\left\langle |\nabla T_{\text{Pe}}|^{2}\right\rangle \leq\left\langle |\nabla\Delta^{-1}\left[\mathbf{u}\cdot\nabla\eta-f\right]|^{2}\right\rangle +\frac{1}{\text{Pe}^{2}}\left\langle |\nabla\mathbf{u}|^{2}\right\rangle \fint_{D}|\nabla\eta|^{2}\quad\forall\,\eta\in H_{0}^{1}(D).
\]
To prove \prettyref{prop:crucial-estimate}, we shall plug in the test
functions $\eta$ from \prettyref{subsec:Branching-flows}, and estimate
the ensuing mess of terms. We handle the ``advection term'' involving
$\mathbf{u}\cdot\nabla\eta-f$ in \prettyref{subsec:The-advection-term}
and the ``enstrophy term'' involving $\left\langle |\nabla\mathbf{u}|^{2}\right\rangle $
in \prettyref{subsec:The-enstrophy-term}.

\subsubsection{The advection term\label{subsec:The-advection-term}}

Averaging the result of \prettyref{lem:decomposition-quadratic-form}
shows that
\[
\left\langle |\nabla\Delta^{-1}\left(\mathbf{u}\cdot\nabla\eta-f\right)|^{2}\right\rangle \leq\left\langle |\overline{\mathbf{u}\cdot\hat{\mathbf{e}}_{r}\eta}-F|^{2}\right\rangle +\left\langle Q(\mathbf{u}\eta-g\hat{\mathbf{e}}_{r})\right\rangle 
\]
where for the reader's convenience we repeat the definition of the
quadratic form: 
\begin{equation}
Q(\mathbf{v})=\min_{\varphi\in H^{1}(D)}\,\fint_{D}|-\frac{1}{r}\partial_{\theta}\varphi+\mathbf{v}\cdot\hat{\mathbf{e}}_{r}-\overline{\mathbf{v}\cdot\hat{\mathbf{e}}_{r}}|^{2}+|\partial_{r}\varphi+\mathbf{v}\cdot\hat{\mathbf{e}}_{\theta}|^{2}\,d\mathbf{x}.\label{eq:quad-form-reminder}
\end{equation}
We bound the first average in \prettyref{lem:mean-flux-avg} and the
second one in \prettyref{lem:quad-form}. 
\begin{lem}
\label{lem:mean-flux-avg}Every branching flow\textendash test function
pair $(\mathbf{u},\eta)$ defined in \prettyref{subsec:Branching-flows}
satisfies
\[
\left\langle |\overline{\mathbf{u}\cdot\hat{\mathbf{e}}_{r}\eta}-F|^{2}\right\rangle \lesssim l_{\emph{bulk}}^{2}\left\langle |\nabla f|^{2}\right\rangle +\delta_{\emph{bl}}\left\langle |f|^{2}\right\rangle .
\]
\end{lem}
\begin{proof}
It suffices to prove the steady analog of the result at a.e.\ time,
i.e., 
\[
\fint_{D}|\overline{\mathbf{u}\cdot\hat{\mathbf{e}}_{r}\eta}-F|^{2}\lesssim l_{\text{bulk}}^{2}\fint_{D}|\nabla f|^{2}+\delta_{\text{bl}}\fint_{D}|f|^{2}.
\]
From the definitions of $\mathbf{u}$ and $\eta$ and our assumption
\prettyref{eq:nearly-diagonal},
\[
\mathbf{u}\eta\cdot\hat{\mathbf{e}}_{r}=\sum_{|k-k'|\leq1}\chi_{k}\chi_{k'}\mathbf{u}_{k}\eta_{k'}\cdot\hat{\mathbf{e}}_{r}
\]
so that
\[
\left|\overline{\mathbf{u}\eta\cdot\hat{\mathbf{e}}_{r}}-\sum_{k=1}^{n}\chi_{k}^{2}F\right|\leq\sum_{k=1}^{n}\chi_{k}^{2}\left|\overline{\mathbf{u}_{k}\eta_{k}\cdot\hat{\mathbf{e}}_{r}}-F\right|+\sum_{\substack{|k-k'|\leq1\\
k\neq k'
}
}|\chi_{k}\chi_{k'}|\left|\overline{\mathbf{u}_{k}\eta_{k'}\cdot\hat{\mathbf{e}}_{r}}\right|.
\]
The first term is handled by \prettyref{eq:flux-bound}, which shows
that
\[
\left|\overline{\mathbf{u}_{k}\eta_{k}\cdot\hat{\mathbf{e}}_{r}}-F\right|\lesssim l_{k}\int_{D_{r}}|\nabla f|.
\]
To control the second term, write using the $L^{2}$-orthogonality
of $\Psi_{k}'(\theta)$ and $\Psi_{k'}'(\theta)$ that
\begin{align*}
\left|\overline{\mathbf{u}_{k}\eta_{k'}\cdot\hat{\mathbf{e}}_{r}}\right| & =\left|\fint_{0}^{2\pi}\partial_{\theta}g\Psi_{k}\Psi_{k'}'+g\Psi_{k}'\Psi_{k'}'\right|\\
 & \leq\left|\fint_{0}^{2\pi}\partial_{\theta}g\Psi_{k}\Psi_{k'}'\right|+\left|\fint_{0}^{2\pi}\partial_{\theta}g\partial_{\theta}^{-1}\left(\Psi_{k}'\Psi_{k'}'\right)\right|\\
 & \lesssim||\partial_{\theta}g||_{L_{\theta}^{1}}(l_{k}\vee l_{k'})\lesssim(l_{k}\vee l_{k'})\int_{D_{r}}|\nabla f|
\end{align*}
where in the last step we used \prettyref{eq:integrated-estimates-on-g-derivs}.
Adding up, 
\[
\left|\overline{\mathbf{u}\eta\cdot\hat{\mathbf{e}}_{r}}-\sum_{k=1}^{n}\chi_{k}^{2}F\right|\lesssim\sum_{|k-k'|\leq1}|\chi_{k}\chi_{k'}|(l_{k}\vee l_{k'})\int_{D_{r}}|\nabla f|
\]
for $r>0$.

Continuing, we have by the triangle and Jensen's inequality that
\begin{align*}
\int_{D}|\overline{\mathbf{u}\eta\cdot\hat{\mathbf{e}}_{r}}-F|^{2}\,d\mathbf{x} & \lesssim\int_{0}^{1}|\overline{\mathbf{u}\eta\cdot\hat{\mathbf{e}}_{r}}-\sum_{k=1}^{n}\chi_{k}^{2}F|^{2}\,rdr+\int_{0}^{1}\left|1-\sum_{k=1}^{n}\chi_{k}^{2}\right|^{2}|F|^{2}\,rdr\\
 & \lesssim\sum_{|k-k'|\leq1}\int_{0}^{1}|\chi_{k}\chi_{k'}|^{2}(l_{k}\vee l_{k'})^{2}r^{3}\int_{D_{r}}|\nabla f|^{2}+\int_{0}^{1}\left|1-\sum_{k=1}^{n}\chi_{k}^{2}\right|^{2}|F|^{2}r\\
 & \lesssim l_{\text{bulk}}^{2}\int_{D}|\nabla f|^{2}+\delta_{\text{bl}}\fint_{r_{\text{bl}}}^{1}|F|^{2}r.
\end{align*}
Recalling the definition of $F=\frac{1}{2\pi r}\int_{D_{r}}f$ from
\prettyref{eq:unsteady-Fandg}, we use Jensen's inequality once more
to get that
\[
\fint_{r_{\text{bl}}}^{1}|F|^{2}r\lesssim\fint_{r_{\text{bl}}}^{1}\left[\int_{D_{r}}|f|^{2}\right]r\lesssim\int_{D}|f|^{2}.
\]
The result follows.
\end{proof}
\begin{lem}
\label{lem:quad-form} Under the assumptions in \prettyref{eq:branching-analysis-assumptions-1},
the branching flow\textendash test function pairs $(\mathbf{u},\eta)$
from \prettyref{subsec:Branching-flows} satisfy
\[
\left\langle Q(\mathbf{u}\eta-g\hat{\mathbf{e}}_{r})\right\rangle \lesssim\left(l_{\emph{bulk}}^{2}+\int_{r_{\emph{bulk}}}^{r_{\emph{bl}}}(\ell'(r))^{2}\,dr+\frac{l_{\emph{bl}}^{2}}{\delta_{\emph{bl}}}\right)\left\langle |f|^{2}+|\nabla f|^{2}\right\rangle +\delta_{\emph{bl}}\left\langle |f|^{2}\right\rangle 
\]
with a constant depending only on those in \prettyref{eq:branching-analysis-assumptions-1}.
\end{lem}
\begin{proof}
Again we argue a.e.\ in time. Begin with the bound
\begin{align*}
Q(\mathbf{u}\eta-g\hat{\mathbf{e}}_{r}) & \lesssim Q(\mathbf{u}\eta-\sum_{k=1}^{n}\chi_{k}^{2}g\hat{\mathbf{e}}_{r})+\int_{D}|\left(\sum_{k=1}^{n}\chi_{k}^{2}-1\right)g|^{2}\\
 & \lesssim Q(\mathbf{u}\eta-\sum_{k=1}^{n}\chi_{k}^{2}g\hat{\mathbf{e}}_{r})+\delta_{\text{bl}}\int_{D}|f|^{2}
\end{align*}
where in the first step we applied the definition \prettyref{eq:quad-form-reminder}
of the quadratic form $Q$, and in the second step we used the formula
$g=\fint_{0}^{r}\rho f$ from \prettyref{eq:unsteady-Fandg} along
with Jensen's inequality and \prettyref{eq:Pythagorean-sum}. Calling
\[
\mathbf{v}=\mathbf{u}\eta-\sum_{k=1}^{n}\chi_{k}^{2}g\hat{\mathbf{e}}_{r}
\]
and using
\[
\varphi=r\partial_{\theta}^{-1}(\mathbf{v}\cdot\hat{\mathbf{e}}_{r})=r\int_{0}^{\theta}\mathbf{v}\cdot\hat{\mathbf{e}}_{r}(r,\theta')-\overline{\mathbf{v}\cdot\hat{\mathbf{e}}_{r}}(r)\,d\theta'
\]
as in \prettyref{rem:good-test-function}, we deduce that
\begin{align*}
Q(\mathbf{v}) & \leq\fint_{D}|\partial_{r}\left(r\partial_{\theta}^{-1}(\mathbf{v}\cdot\hat{\mathbf{e}}_{r})\right)+\mathbf{v}\cdot\hat{\mathbf{e}}_{\theta}|^{2}\\
 & \lesssim\fint_{D}|\partial_{\theta}^{-1}(\mathbf{v}\cdot\hat{\mathbf{e}}_{r})|^{2}+\fint_{D}r^{2}|\partial_{r}\partial_{\theta}^{-1}(\mathbf{v}\cdot\hat{\mathbf{e}}_{r})|^{2}+\fint_{D}|\mathbf{v}\cdot\hat{\mathbf{e}}_{\theta}|^{2}.
\end{align*}
We estimate these integrals one-by-one.

For the first, note using the definitions from \prettyref{subsec:Branching-flows}
and in particular \prettyref{eq:nearly-diagonal} that
\begin{align*}
\partial_{\theta}^{-1}(\mathbf{v}\cdot\hat{\mathbf{e}}_{r}) & =\partial_{\theta}^{-1}(\mathbf{u}\eta\cdot\hat{\mathbf{e}}_{r}-\sum_{k=1}^{n}\chi_{k}^{2}g)=\sum_{|k-k'|\leq1}\chi_{k}\chi_{k'}\partial_{\theta}^{-1}\left[\mathbf{u}_{k}\eta_{k'}\cdot\hat{\mathbf{e}}_{r}-\delta_{kk'}g\right]\\
 & =\sum_{|k-k'|\leq1}\chi_{k}\chi_{k'}\partial_{\theta}^{-1}\left[\partial_{\theta}g\Psi_{k}\Psi_{k'}'+g\left(\Psi_{k}'\Psi_{k'}'-\delta_{kk'}\right)\right]
\end{align*}
where $\delta_{kk'}$ is the Kronecker delta function (one if $k=k'$,
zero otherwise). Arguing as in the proof of \prettyref{lem:building-block}\textemdash see
the paragraph leading up to \prettyref{eq:final-Lebesgue1}\textemdash we
write that
\begin{align}
||\partial_{\theta}^{-1}\left[\partial_{\theta}g\Psi_{k}\Psi_{k'}'+g\left(\Psi_{k}'\Psi_{k'}'-\delta_{kk'}\right)\right]||_{L_{\theta}^{1}} & \lesssim(l_{k}\vee l_{k'})\left(||g||_{L_{\theta}^{1}}+||\partial_{\theta}g||_{L_{\theta}^{1}}\right)\label{eq:applied-Lebesgue1}\\
 & \lesssim(l_{k}\vee l_{k'})\left(\frac{1}{r}\int_{D_{r}}|f|+\int_{D_{r}}|\nabla f|\right)\nonumber 
\end{align}
where in the last step we used \prettyref{eq:integrated-estimates-on-g}
and \prettyref{eq:integrated-estimates-on-g-derivs}. Squaring and
integrating, there follows
\begin{align*}
\int_{D}|\partial_{\theta}^{-1}(\mathbf{v}\cdot\hat{\mathbf{e}}_{r})|^{2} & \lesssim\sum_{|k-k'|\leq1}\int_{0}^{1}|\chi_{k}\chi_{k'}|^{2}(l_{k}\vee l_{k'})^{2}\left(\int_{D_{r}}|f|^{2}+r^{2}\int_{D_{r}}|\nabla f|^{2}\right)\,rdr\\
 & \lesssim l_{\text{bulk}}^{2}\int_{D}|f|^{2}+|\nabla f|^{2}.
\end{align*}

Continuing with the second integral, write using the product rule
that
\begin{align*}
\partial_{r}\partial_{\theta}^{-1}(\mathbf{v}\cdot\hat{\mathbf{e}}_{r}) & =\sum_{|k-k'|\leq1}(\chi_{k}\chi_{k'})'\partial_{\theta}^{-1}\left[\partial_{\theta}g\Psi_{k}\Psi_{k'}'+g\left(\Psi_{k}'\Psi_{k'}'-\delta_{kk'}\right)\right]\\
 & \qquad+\sum_{|k-k'|\leq1}\chi_{k}\chi_{k'}\partial_{r}\partial_{\theta}^{-1}\left[\partial_{\theta}g\Psi_{k}\Psi_{k'}'+g\left(\Psi_{k}'\Psi_{k'}'-\delta_{kk'}\right)\right]
\end{align*}
and note the estimate
\begin{align}
||\partial_{r}\partial_{\theta}^{-1}\left[\partial_{\theta}g\Psi_{k}\Psi_{k'}'+g\left(\Psi_{k}'\Psi_{k'}'-\delta_{kk'}\right)\right]||_{L_{\theta}^{1}} & \lesssim(l_{k}\vee l_{k'})\left(||\partial_{r}g||_{L_{\theta}^{1}}+||\partial_{r\theta}g||_{L_{\theta}^{1}}\right)\label{eq:applied-Lebesgue2}\\
 & \lesssim(l_{k}\vee l_{k'})\left(\frac{1}{r^{2}}\int_{D_{r}}|f|+||f||_{L_{\theta}^{1}}+\frac{1}{r}\int_{D_{r}}|\nabla f|+r||\nabla f||_{L_{\theta}^{1}}\right)\nonumber 
\end{align}
holds in addition to \prettyref{eq:applied-Lebesgue1}. Its proof
is essentially the same, and again we point the reader to the paragraph
leading up to \prettyref{eq:final-Lebesgue2} for the details. Using
both \prettyref{eq:applied-Lebesgue1} and \prettyref{eq:applied-Lebesgue2},
we get that
\begin{align*}
&\int_{D}r^{2}|\partial_{r}\partial_{\theta}^{-1}(\mathbf{v}\cdot\hat{\mathbf{e}}_{r})|^{2} \\ 
&\quad \lesssim\sum_{|k-k'|\leq1}\int_{0}^{1}|(\chi_{k}\chi_{k'})'|^{2}(l_{k}\vee l_{k'})^{2}\left(r^{2}\int_{D_{r}}|f|^{2}+r^{4}\int_{D_{r}}|\nabla f|^{2}\right)\,rdr\\
&\quad\qquad+\sum_{|k-k'|\leq1}\int_{0}^{1}|\chi_{k}\chi_{k'}|^{2}(l_{k}\vee l_{k'})^{2}\left(\int_{D_{r}}|f|^{2}+r^{2}\int_{D_{r}}|\nabla f|^{2}+r^{2}||f||_{L_{\theta}^{2}}^{2}+r^{4}||\nabla f||_{L_{\theta}^{2}}^{2}\right)\,rdr\\
&\quad \lesssim\left(\int_{r_{\text{bulk}}}^{r_{\text{bl}}}(\ell'(r))^{2}\,dr+\frac{l_{\text{bl}}^{2}}{\delta_{\text{bl}}}+l_{\text{bulk}}^{2}\right)\int_{D}|f|^{2}+|\nabla f|^{2}.
\end{align*}
Note we used the definition of $\ell$ and our hypotheses that $|l_{k+1}-l_{k}|\sim l_{k+1}\sim l_{k}$
and $\delta_{k+1}\sim\delta_{k}$ to bring in $\ell'$.

Finally, since
\begin{align*}
\mathbf{v}\cdot\hat{\mathbf{e}}_{\theta} & =\mathbf{u}\eta\cdot\hat{\mathbf{e}}_{\theta}=\sum_{|k-k'|\leq1}(\chi_{k}'\psi_{k}+\chi_{k}\mathbf{u}_{k}\cdot\hat{\mathbf{e}}_{\theta})\chi_{k'}\eta_{k'}\\
 & =-\sum_{|k-k'|\leq1}\chi_{k}'\chi_{k'}rg\Psi_{k}\Psi_{k'}'+\chi_{k}\chi_{k'}rf\Psi_{k}\Psi_{k'}'
\end{align*}
we get by a completely analogous argument that 
\begin{align*}
\int_{D}|\mathbf{v}\cdot\hat{\mathbf{e}}_{\theta}|^{2} & \lesssim\sum_{|k-k'|\leq1}\int_{0}^{1}\left(|\chi_{k}'\chi_{k'}|^{2}l_{k}^{2}r^{2}||g||_{L_{\theta}^{2}}^{2}+|\chi_{k}\chi_{k'}|^{2}l_{k}^{2}r^{2}||f||_{L_{\theta}^{2}}^{2}\right)\,rdr\\
 & \lesssim\left(\int_{r_{\text{bulk}}}^{r_{\text{bl}}}(\ell'(r))^{2}\,dr+\frac{l_{\text{bl}}^{2}}{\delta_{\text{bl}}}+l_{\text{bulk}}^{2}\right)\int_{D}|f|^{2}.
\end{align*}
In the last step, we applied Jensen's inequality with $g=\fint_{0}^{r}\rho f$.
Adding up the estimates and averaging in time proves the result.
\end{proof}
Combining \prettyref{lem:mean-flux-avg} and \prettyref{lem:quad-form}
with our hypothesis that $l_{\text{bl}}\lesssim\delta_{\text{bl}}$
from \prettyref{eq:branching-analysis-assumptions-2} proves the first
part of the estimate in \prettyref{prop:crucial-estimate}.

\subsubsection{The enstrophy term\label{subsec:The-enstrophy-term}}

We now estimate the gradients of $\mathbf{u}$ and $\eta$. 
\begin{lem}
\label{lem:enstrophy-estimates}Under the assumptions in \prettyref{eq:branching-analysis-assumptions-1}
and \prettyref{eq:branching-analysis-assumptions-2}, the branching
flow\textendash test function pairs $(\mathbf{u},\eta)$ from \prettyref{subsec:Branching-flows}
satisfy
\begin{align*}
\left\langle |\nabla\mathbf{u}|^{2}\right\rangle  & \lesssim\left(\frac{1}{l_{\emph{bulk}}^{2}}+\int_{r_{\emph{bulk}}}^{r_{\emph{bl}}}\frac{1}{(\ell(r))^{2}}\,dr+\frac{\delta_{\emph{bl}}}{l_{\emph{bl}}^{2}}\right)\left\langle |f|^{2}+|\nabla f|^{2}+|\nabla\nabla f|^{2}\right\rangle ,\\
\left\langle |\nabla\eta|^{2}\right\rangle  & \lesssim\frac{1}{l_{\emph{bulk}}^{2}}+\int_{r_{\emph{bulk}}}^{r_{\emph{bl}}}\frac{1}{(\ell(r))^{2}}\,dr+\frac{\delta_{\emph{bl}}}{l_{\emph{bl}}^{2}}.
\end{align*}
The constants implicit in these estimates depend only on those in
\prettyref{eq:branching-analysis-assumptions-1} and \prettyref{eq:branching-analysis-assumptions-2}.
\end{lem}
\begin{proof}
We start with the formulas
\begin{align*}
\mathbf{u} & =\sum_{k=1}^{n}\chi_{k}\mathbf{u}_{k}+\chi_{k}'\psi_{k}\hat{\mathbf{e}}_{\theta},\\
\nabla\mathbf{u} & =\sum_{k=1}^{n}\chi_{k}\nabla\mathbf{u}_{k}+\chi_{k}'(\mathbf{u}_{k}\otimes\hat{\mathbf{e}}_{r}-\hat{\mathbf{e}}_{\theta}\otimes\mathbf{u}_{k}^{\perp})+\chi_{k}''\psi_{k}\hat{\mathbf{e}}_{\theta}\otimes\hat{\mathbf{e}}_{r},
\end{align*}
which follow from the definitions in \prettyref{subsec:Branching-flows}
and the fact that $\mathbf{u}_{k}^{\perp}=-\nabla\psi_{k}$ as $(\cdot)^{\perp}$
is a counterclockwise rotation by $\pi/2$. Looking back at the estimates
in \prettyref{eq:velocity-bd} and \prettyref{eq:enstrophy-bd}, we
see that
\begin{align*}
|\nabla\mathbf{u}| & \lesssim\sum_{k=1}^{n}|\chi_{k}||\nabla\mathbf{u}_{k}|+|\chi_{k}'||\mathbf{u}_{k}|+|\chi_{k}''||\psi_{k}|\\
 & \lesssim\sum_{k=1}^{n}|\chi_{k}|\left(\frac{A_{1}}{rl_{k}}+\frac{A_{2}}{l_{k}}+rl_{k}A_{3}+|f|+r|\nabla f|\right)\\
 & \qquad+|\chi_{k}'|\left(A_{1}+rl_{k}A_{2}+rl_{k}|f|\right)+|\chi_{k}''|rl_{k}A_{1}
\end{align*}
where 
\[
A_{1}=\frac{1}{r}\int_{0}^{r}\rho|f|\,d\rho,\quad A_{2}=\frac{1}{r}\int_{0}^{r}\rho|\nabla f|\,d\rho,\quad A_{3}=\frac{1}{r}\int_{0}^{r}\rho|\nabla\nabla f|\,d\rho.
\]
Squaring and integrating, there follows
\begin{align*}
\int_{D}|\nabla\mathbf{u}|^{2} & \lesssim\sum_{k=1}^{n}\int_{0}^{1}|\chi_{k}|^{2}\left(\frac{1}{r^{2}l_{k}^{2}}||A_{1}||_{L_{\theta}^{2}}^{2}+\frac{1}{l_{k}^{2}}||A_{2}||_{L_{\theta}^{2}}^{2}+r^{2}l_{k}^{2}||A_{3}||_{L_{\theta}^{2}}^{2}+||f||_{L_{\theta}^{2}}^{2}+r^{2}||\nabla f||_{L_{\theta}^{2}}^{2}\right)\,rdr\\
 & \qquad+\sum_{k=1}^{n}\int_{0}^{1}|\chi_{k}'|^{2}\left(||A_{1}||_{L_{\theta}^{2}}^{2}+r^{2}l_{k}^{2}||A_{2}||_{L_{\theta}^{2}}^{2}+r^{2}l_{k}^{2}||f||_{L_{\theta}^{2}}^{2}\right)\,rdr+\sum_{k=1}^{n}\int_{0}^{1}|\chi_{k}''|^{2}r^{2}l_{k}^{2}||A_{1}||_{L_{\theta}^{2}}^{2}\,rdr\\
 & :=I_{1}+I_{2}+I_{3}.
\end{align*}
We bound these three sums in turn, using the estimates
\begin{align}
||A_{1}||_{L_{\theta}^{2}}^{2} & \leq\fint_{0}^{r}\rho^{2}||f||_{L_{\theta}^{2}}^{2}\,d\rho\lesssim\int_{D_{r}}|f|^{2}\,d\mathbf{x},\label{eq:estimate-on-A1}\\
||A_{2}||_{L_{\theta}^{2}}^{2} & \leq\fint_{0}^{r}\rho^{2}||\nabla f||_{L_{\theta}^{2}}^{2}\,d\rho\lesssim\int_{D_{r}}|\nabla f|^{2}\,d\mathbf{x},\label{eq:estimate-on-A2}\\
||A_{3}||_{L_{\theta}^{2}}^{2} & \leq\fint_{0}^{r}\rho^{2}||\nabla\nabla f||_{L_{\theta}^{2}}^{2}\,d\rho\lesssim\int_{D_{r}}|\nabla\nabla f|^{2}\,d\mathbf{x}.\label{eq:estimate-on-A3}
\end{align}
Note these follow from the definitions via Jensen's inequality.

The first sum $I_{1}$ involves the cutoff functions $\{\chi_{k}\}$
directly. Focusing on $r\in(r_{\text{bulk}},r_{\text{bl}})$ for now,
which are always larger than $1/2$ and no larger than $1$, we estimate
\begin{align*}
 & \sum_{k=1}^{n}\int_{r_{\text{bulk}}}^{r_{\text{bl}}}|\chi_{k}|^{2}\left(\frac{1}{r^{2}l_{k}^{2}}||A_{1}||_{L_{\theta}^{2}}^{2}+\frac{1}{l_{k}^{2}}||A_{2}||_{L_{\theta}^{2}}^{2}+r^{2}l_{k}^{2}||A_{3}||_{L_{\theta}^{2}}^{2}\right)\,rdr\\
 & \quad\lesssim\sum_{k=1}^{n}\int_{r_{\text{bulk}}}^{r_{\text{bl}}}\frac{|\chi_{k}|^{2}}{l_{k}^{2}}\left(||A_{1}||_{L_{\theta}^{2}}^{2}+||A_{2}||_{L_{\theta}^{2}}^{2}+||A_{3}||_{L_{\theta}^{2}}^{2}\right)\\
 & \quad\lesssim\int_{r_{\text{bulk}}}^{r_{\text{bl}}}\frac{1}{(\ell(r))^{2}}\,dr\cdot\int_{D}|f|^{2}+|\nabla f|^{2}+|\nabla\nabla f|^{2}
\end{align*}
by our definition of $\ell$ and the hypothesis that $l_{k+1}\sim l_{k}$.
For $r\in(0,r_{\text{bulk}})$, which belong only to the support of
$\chi_{1}$, we use Fubini's theorem and the first parts of \prettyref{eq:estimate-on-A1}-\prettyref{eq:estimate-on-A3}
to write that
\begin{align*}
 & \int_{0}^{r_{\text{bulk}}}|\chi_{1}|^{2}\left(\frac{1}{r^{2}l_{1}^{2}}||A_{1}||_{L_{\theta}^{2}}^{2}+\frac{1}{l_{1}^{2}}||A_{2}||_{L_{\theta}^{2}}^{2}+r^{2}l_{1}^{2}||A_{3}||_{L_{\theta}^{2}}^{2}\right)\,rdr\\
 & \quad\lesssim\int_{0\leq\rho\leq r\leq r_{\text{bulk}}}\left[\frac{\rho^{2}}{r^{2}l_{\text{bulk}}^{2}}||f(\rho,\cdot)||_{L_{\theta}^{2}}^{2}+\frac{\rho^{2}}{l_{\text{bulk}}^{2}}||\nabla f(\rho,\cdot)||_{L_{\theta}^{2}}^{2}+\frac{l_{\text{bulk}}^{2}\rho^{2}}{r^{2}}||\nabla\nabla f(\rho,\cdot)||_{L_{\theta}^{2}}^{2}\right]\,d\rho dr\\
 & \quad\lesssim\frac{1}{l_{\text{bulk}}^{2}}\int_{D}|f|^{2}+|\nabla f|^{2}+|\nabla\nabla f|^{2}
\end{align*}
similarly to what we did at the very end of \prettyref{sec:Steady-branching-flows}.
Finally, 
\[
\int_{r_{\text{bl}}}^{1}|\chi_{n}|^{2}\left(\frac{1}{r^{2}l_{n}^{2}}||A_{1}||_{L_{\theta}^{2}}^{2}+\frac{1}{l_{n}^{2}}||A_{2}||_{L_{\theta}^{2}}^{2}+r^{2}l_{n}^{2}||A_{3}||_{L_{\theta}^{2}}^{2}\right)\,rdr\lesssim\frac{\delta_{\text{bl}}}{l_{\text{bl}}^{2}}\int_{D}|f|^{2}+|\nabla f|^{2}+|\nabla\nabla f|^{2}.
\]
The bound
\[
\sum_{k=1}^{n}\int_{0}^{1}|\chi_{k}|^{2}\left(||f||_{L_{\theta}^{2}}^{2}+r^{2}||\nabla f||_{L_{\theta}^{2}}^{2}\right)\,rdr\lesssim\int_{D}|f|^{2}+|\nabla f|^{2}
\]
is clear. Altogether, 
\[
I_{1}\lesssim\left(\frac{1}{l_{\text{bulk}}^{2}}+\int_{r_{\text{bulk}}}^{r_{\text{bl}}}\frac{1}{(\ell(r))^{2}}\,dr+\frac{\delta_{\text{bl}}}{l_{\text{bl}}^{2}}\right)\cdot\int_{D}|f|^{2}+|\nabla f|^{2}+|\nabla\nabla f|^{2}.
\]

The second sum $I_{2}$ involves the derivatives $\{\chi_{k}'\}$.
This time, we only need to handle $r>r_{\text{bulk}}$. Using the
second parts of \prettyref{eq:estimate-on-A1} and \prettyref{eq:estimate-on-A2},
we get that
\begin{align*}
I_{2} & =\sum_{k=1}^{n}\int_{r_{\text{bulk}}}^{1}|\chi_{k}'|^{2}\left(||A_{1}||_{L_{\theta}^{2}}^{2}+r^{2}l_{k}^{2}||A_{2}||_{L_{\theta}^{2}}^{2}+r^{2}l_{k}^{2}||f||_{L_{\theta}^{2}}^{2}\right)\,rdr\\
 & \lesssim\left(\sum_{k=1}^{n-1}\int_{r_{k}}^{r_{k+1}}\frac{1}{(\delta_{k}\wedge\delta_{k+1})^{2}}+\int_{r_{\text{bl}}}^{1}\frac{1}{\delta_{\text{bl}}^{2}}\right)\cdot\int_{D}|f|^{2}+|\nabla f|^{2}\\
 & \lesssim\left(\int_{r_{\text{bulk}}}^{r_{\text{bl}}}\frac{1}{(\ell(r))^{2}}\,dr+\frac{\delta_{\text{bl}}}{l_{\text{bl}}^{2}}\right)\cdot\int_{D}|f|^{2}+|\nabla f|^{2}.
\end{align*}
The bounds
\[
\frac{1}{\delta_{k}\wedge\delta_{k+1}}\lesssim\frac{1}{\delta_{k+1}}\lesssim\frac{1}{l_{k+1}}\lesssim\frac{1}{\ell(r)}\quad\forall\,r\in(r_{k},r_{k+1})
\]
were used to pass between the second and third lines in the estimate
above. They hold by our assumptions that $\delta_{k}\sim\delta_{k+1}$,
$l_{k}\lesssim\delta_{k}$, and $l_{k}\sim l_{k+1}$ along with the
definition of $\ell$. 

Finally, we estimate $I_{3}$ which involves the second derivatives
$\{\chi_{k}''\}$. Again we use \prettyref{eq:estimate-on-A1} to
write that
\begin{align*}
I_{3} & =\sum_{k=1}^{n}\int_{r_{\text{bulk}}}^{1}|\chi_{k}''|^{2}r^{2}l_{k}^{2}||A_{1}||_{L_{\theta}^{2}}^{2}\,rdr\\
 & \lesssim\left(\sum_{k=1}^{n-1}\int_{r_{k}}^{r_{k+1}}\frac{(l_{k}\vee l_{k+1})^{2}}{(\delta_{k}\wedge\delta_{k+1})^{4}}+\int_{r_{\text{bl}}}^{1}\frac{l_{\text{bl}}^{2}}{\delta_{\text{bl}}^{4}}\right)\cdot\int_{D}|f|^{2}\\
 & \lesssim\left(\int_{r_{\text{bulk}}}^{r_{\text{bl}}}\frac{1}{(\ell(r))^{2}}\,dr+\frac{\delta_{\text{bl}}}{l_{\text{bl}}^{2}}\right)\cdot\int_{D}|f|^{2}.
\end{align*}
Adding up the estimates on $I_{1}$, $I_{2}$, and $I_{3}$ and averaging
in time gives the first part of the claim.

A similar, and much simpler, argument proves the desired bound on
$\nabla\eta$. In particular, by the definition of $\eta$ in \prettyref{subsec:Branching-flows}
and the estimates from \prettyref{eq:eta-estimates}, there holds
\[
|\nabla\eta|\leq\sum_{k=1}^{n}|\chi_{k}||\nabla\eta_{k}|+|\chi_{k}'||\eta_{k}|\lesssim\sum_{k=1}^{n}|\chi_{k}|\frac{1}{rl_{k}}+|\chi_{k}'|.
\]
Hence
\[
\int_{D}|\nabla\eta|^{2}\lesssim\sum_{k=1}^{n}\int_{0}^{1}\left(|\chi_{k}|^{2}\frac{1}{r^{2}l_{k}^{2}}+|\chi_{k}'|^{2}\right)\,rdr.
\]
We recognize terms like those from the bound on
$\nabla\mathbf{u}$, with the difference being that the terms involving
$f$ are now replaced by the number one. The same manipulations apply
as before.
\end{proof}
Combining \prettyref{lem:enstrophy-estimates} with what we proved
in \prettyref{subsec:The-advection-term} yields the rest of \prettyref{prop:crucial-estimate}. 

\subsection{Optimal branching flows\label{subsec:Optimal-branching-flows}}

Finally, we optimize over our branching flows to prove the upper bound
in \prettyref{thm:main-theorem}. Recall from \prettyref{prop:crucial-estimate}
that the (rescaled) velocities $\lambda_{\text{Pe}}\mathbf{u}$ and temperature fields $T_{\text{Pe}}$ achieve 
\begin{equation}
\left\langle |\nabla T_{\text{Pe}}|^{2}\right\rangle \lesssim C_{0}(f)\cdot M(\{l_{k}\}_{k=1}^{n},\{r_{k}\}_{k=1}^{n},n;\text{Pe})\quad\text{and}\quad\left\langle |\nabla(\lambda_{\text{Pe}}\mathbf{u})|^{2}\right\rangle =\text{Pe}^{2}\label{eq:upper-bound-on-temperature}
\end{equation}
where 
\begin{align*}
C_{0} & =\left\langle |f|^{2}+|\nabla f|^{2}+|\nabla\nabla f|^{2}\right\rangle ,\\
M & =l_{\text{bulk}}^{2}+\int_{r_{\text{bulk}}}^{r_{\text{bl}}}(\ell'(r))^{2}\,dr+\delta_{\text{bl}}+\frac{1}{\text{Pe}^{2}}\left[\frac{1}{l_{\text{bulk}}^{2}}+\int_{r_{\text{bulk}}}^{r_{\text{bl}}}\frac{1}{(\ell(r))^{2}}\,dr+\frac{\delta_{\text{bl}}}{l_{\text{bl}}^{2}}\right]^{2}.
\end{align*}
We find it useful to work directly with the continuous ``scale function''
$\ell(r)$ introduced in \prettyref{eq:continuous-scale}, and to
wait to enforce the interpolation rule
\begin{equation}
l_{k}=\ell(r_{k})\quad\text{for }k=1,\dots,n\label{eq:interpolation-rule}
\end{equation}
until it comes time to select the parameters $\{l_{k}\}_{k=1}^{n}$
and $\{r_{k}\}_{k=1}^{n}$.

Recall $l_{1}=l_{\text{bulk}}$, $l_{n}=l_{\text{bl}}$, $r_{1}=r_{\text{bulk}}$,
and $r_{n}=r_{\text{bl}}$, and consider the one-dimensional variational
problem
\[
\min_{\substack{\ell(r)\\
\ell(r_{\text{bulk}})=l_{\text{bulk}}\\
\ell(r_{\text{bl}})=l_{\text{bl}}
}
}\,\int_{r_{\text{bulk}}}^{r_{\text{bl}}}(\ell')^{2}\,dr+\frac{1}{\text{Pe}^{2}}\left(\int_{r_{\text{bulk}}}^{r_{\text{bl}}}\frac{1}{\ell^{2}}\,dr\right)^{2}
\]
suggested by minimizing $M$. Its solution obeys
\[
(\ell')^{2}\sim\frac{1}{\text{Pe}^{2}}\left(\int_{r_{\text{bulk}}}^{r_{\text{bl}}}\frac{1}{\ell^{2}}\,dr\right)\frac{1}{\ell^{2}}.
\]
Setting $\ell(1)=0$ and integrating yields
\begin{equation}
\ell(r)\sim\frac{1}{\text{Pe}^{1/2}}\left(\int_{r_{\text{bulk}}}^{r_{\text{bl}}}\frac{1}{\ell^{2}}\,dr\right)^{1/4}\sqrt{1-r}.\label{eq:optimal-balance}
\end{equation}
Put another way, 
\[
\ell(r)=c(\text{Pe})\sqrt{1-r}.
\]

The constant $c(\text{Pe})$ and the parameters $r_{\text{bulk}}$, $r_{\text{bl}}$,
$l_{\text{bulk}}$, and $l_{\text{bl}}$ must be determined. For $c$,
note that
\[
\int_{r_{\text{bulk}}}^{r_{\text{bl}}}\frac{1}{\ell^{2}}\,dr=\frac{1}{c^{2}}\int_{r_{\text{bulk}}}^{r_{\text{bl}}}\frac{dr}{1-r}=\frac{1}{c^{2}}\log\left(\frac{\delta_{\text{bulk}}}{\delta_{\text{bl}}}\right)
\]
where $\delta_{\text{bulk}}=1-r_{\text{bulk}}$ and $\delta_{\text{bl}}=1-r_{\text{bl}}$.
Setting this into \prettyref{eq:optimal-balance} and rearranging,
we get that
\[
c(\text{Pe})\sim\frac{1}{\text{Pe}^{1/3}}\log^{1/6}\left(\frac{\delta_{\text{bulk}}}{\delta_{\text{bl}}}\right).
\]
By \prettyref{eq:interpolation-rule},
\[
l_{\text{bulk}}=\ell(r_{\text{bulk}})\sim\frac{\delta_{\text{bulk}}^{1/2}}{\text{Pe}^{1/3}}\log^{1/6}\left(\frac{\delta_{\text{bulk}}}{\delta_{\text{bl}}}\right)\quad\text{and}\quad l_{\text{bl}}=\ell(r_{\text{bl}})\sim\frac{\delta_{\text{bl}}^{1/2}}{\text{Pe}^{1/3}}\log^{1/6}\left(\frac{\delta_{\text{bulk}}}{\delta_{\text{bl}}}\right).
\]

All that remains is to choose $\delta_{\text{bulk}}$ and $\delta_{\text{bl}}$.
Note the quantity $M$ from \prettyref{eq:upper-bound-on-temperature}
satisfies
\begin{align*}
M & \lesssim l_{\text{bulk}}^{2}+\frac{1}{\text{Pe}^{2}}\frac{1}{l_{\text{bulk}}^{4}}+\delta_{\text{bl}}+\frac{1}{\text{Pe}^{2}}\frac{\delta_{\text{bl}}^{2}}{l_{\text{bl}}^{4}}+\frac{1}{\text{Pe}^{2/3}}\log^{4/3}\left(\frac{\delta_{\text{bulk}}}{\delta_{\text{bl}}}\right)\\
 & \lesssim\delta_{\text{bl}}+\frac{1}{\text{Pe}^{2/3}}\left[\frac{1}{\delta_{\text{bulk}}^{2}\log^{2/3}\left(\frac{\delta_{\text{bulk}}}{\delta_{\text{bl}}}\right)}+\log^{4/3}\left(\frac{\delta_{\text{bulk}}}{\delta_{\text{bl}}}\right)\right].
\end{align*}
Minimizing with
\[
\delta_{\text{bulk}}\sim1\quad\text{and}\quad\delta_{\text{bl}}\sim\frac{\log^{1/3}\text{Pe}}{\text{Pe}^{2/3}}
\]
yields the desired bound
\begin{equation}
\left\langle |\nabla T_{\text{Pe}}|^{2}\right\rangle \leq C'(f)\cdot\frac{\log^{4/3}\text{Pe}}{\text{Pe}^{2/3}}\label{eq:desired-upper-bound}
\end{equation}
where $C'\lesssim\left\langle |f|^{2}+|\nabla f|^{2}+|\nabla\nabla f|^{2}\right\rangle $.
Regarding its scaling in $\text{Pe}$, this is the best upper bound achievable by a roll-based
branching flow. 

To complete the proof of \prettyref{thm:main-theorem},
we only need verify that our choices for $\{l_{k}\}_{k=1}^{n}$,
$\{r_{k}\}_{k=1}^{n}$, and $n$ are actually admissible in our analysis
of branching flows. We do this now.

\begin{proof}[Proof of the upper bound from \prettyref{thm:main-theorem}]
To be absolutely clear, we fix
\begin{equation}
\ell(r)=\frac{\log^{1/6}\text{Pe}}{\text{Pe}^{1/3}}\sqrt{1-r},\quad r\in(\frac{1}{2},1)\label{eq:definite-choice-for-lengths}
\end{equation}
for the remainder of the proof, and let $n\in\mathbb{N}$ satisfy 
\[
n\leq\log_{2}\frac{\text{Pe}^{1/3}}{\log^{1/6}\text{Pe}}\leq n+1.
\]
Implicit in this is the requirement that $2\leq \text{Pe}^{1/3}/\log^{1/6}\text{Pe}$.
Define the scales $\{l_{k}\}_{k=1}^{n}$ by taking
\[
l_{k}=\frac{l_{\text{bulk}}}{2^{k-1}}\quad\text{for }k=1,\dots,n
\]
where $l_{\text{bulk}}^{-1}\in\mathbb{N}$ obeys
\[
2\frac{\text{Pe}^{1/3}}{\log^{1/6}\text{Pe}}\leq\frac{1}{l_{\text{bulk}}}\leq4\frac{\text{Pe}^{1/3}}{\log^{1/6}\text{Pe}}.
\]
For the points $\{r_{k}\}_{k=1}^{n}$, enforce the interpolation rule
\[
l_{k}=\ell(r_{k})
\]
which says here that
\begin{equation}
\frac{l_{\text{bulk}}}{2^{k-1}}=\frac{\log^{1/6}\text{Pe}}{\text{Pe}^{1/3}}\sqrt{1-r_{k}}\quad\text{for }k=1,\dots,n.\label{eq:choice-for-points}
\end{equation}
At this point, all available choices have been made and we can go
ahead with our proof of the desired bound \prettyref{eq:desired-upper-bound}.
Actually, we already did most of the heavy lifting in the paragraphs
above, where we explained how these choices follow from optimizing
the result of \prettyref{prop:crucial-estimate}. All that remains
is to verify the hypotheses from \prettyref{subsec:Branching-flows}
and \prettyref{prop:crucial-estimate}. 

First, we check the assumptions of \prettyref{subsec:Branching-flows}.
Clearly, $l_{k}$ decreases with increasing $k$ per \prettyref{eq:lengths-decrease}.
And as $r\mapsto\ell(r)$ is strictly decreasing, the points $r_{k}$
increase with increasing $k$. Taking $k=1$ in \prettyref{eq:choice-for-points},
we get that
\[
1=\frac{1}{l_{\text{bulk}}}\frac{\log^{1/6}\text{Pe}}{\text{Pe}^{1/3}}\sqrt{1-r_{1}}\geq2\sqrt{1-r_{1}}
\]
by our choice of $l_{\text{bulk}}$. So, $r_{1}\geq3/4>1/2$ as per \prettyref{eq:points-increase}. We have shown that our choices for $\{l_{k}\}$,
$\{r_{k}\}$, and $n$ constitute a viable branching flow.

Next, we check the hypotheses of \prettyref{prop:crucial-estimate}.
Its first one \prettyref{eq:branching-analysis-assumptions-1} requires
that $|l_{k+1}-l_{k}|\sim l_{k+1}\sim l_{k}$ and $\delta_{k+1}\sim\delta_{k}$
with fixed numerical constants. The former holds by the dyadic nature
of $l_{k}$. For the latter, introduce the inverse $\ell\mapsto r(\ell)$,
which is strictly decreasing, and note that $\delta_{k}=r_{k+1}-r_{k}$
obeys
\[
\delta_{k}=\int_{l_{k+1}}^{l_{k}}|r'(\ell)|\,d\ell.
\]
It suffices to check that $|r'(\ell_{k})|\sim|r'(\ell_{k+1})|$. Differentiating
\prettyref{eq:definite-choice-for-lengths} implicitly and rearranging
gives 
\begin{equation}
r'(\ell)=-2\frac{\text{Pe}^{1/3}}{\log^{1/6}\text{Pe}}\sqrt{1-r}.\label{eq:implicit-differentiation}
\end{equation}
One sees from \prettyref{eq:choice-for-points} that $1-r_{k}\sim1-r_{k+1}$. Hence, \prettyref{eq:branching-analysis-assumptions-1}
is proved.

Finally, we verify the second hypothesis \prettyref{eq:branching-analysis-assumptions-2}
of \prettyref{prop:crucial-estimate}, which is that $l_{k}\lesssim\delta_{k}$
for each $k$. Since $\delta_{k}=r_{k+1}-r_{k}$ and $l_{k}\sim|l_{k+1}-l_{k}|$,
we must show that 
\[
1\lesssim\left|\frac{r_{k+1}-r_{k}}{l_{k+1}-l_{k}}\right|.
\]
Referring again to $r(\ell)$, it suffices to check that $1\lesssim|r'(\ell)|$
for $\ell\in(l_{1},l_{n})$. Evidently by \prettyref{eq:implicit-differentiation},
\[
|r'(\ell)|\gtrsim\frac{\text{Pe}^{1/3}}{\log^{1/6}\text{Pe}}\sqrt{1-r}\geq\frac{\text{Pe}^{1/3}}{\log^{1/6}\text{Pe}}\sqrt{1-r_{n}}\sim1.
\]
In the last step we used \prettyref{eq:choice-for-points} with $k=n$.
That $l_{\text{bl}}\sim\delta_{\text{bl}}$ is clear. \prettyref{thm:main-theorem}
is proved. \end{proof}

\section{Unsteady roll-like flows for energy-constrained cooling\label{sec:Unsteady-roll-type-flows-energy-constrained}}

We close by proving the upper bound on energy-constrained cooling
from the introduction, by estimating the mean-square temperature gradient of a well-chosen family of convection roll-like flows (see the middle row of \prettyref{fig:flows}). 
It is an open challenge to decide whether
this bound is sharp in its scaling with respect to $\text{Pe}$, in the advective
limit $\text{Pe}\to\infty$. 
In particular, we note the significant gap between the lower bound in \prettyref{prop:lower-bound-energy-constraint} and the upper bound achieved below.
\begin{prop}
\label{prop:upper-bound-energy-constraint}Let $f(\mathbf{x},t)$
satisfy 
\[
\lim_{\tau\to\infty}\,\frac{1}{\sqrt{\tau}}\int_{0}^{\tau}e^{-\lambda_{1}\left((\tau-t)\wedge t\right)}||f(\cdot,t)||_{L^{2}(D)}\,dt=0\quad\text{and}\quad\left\langle |f|^{2}+|\nabla f|^{2}\right\rangle <\infty.
\]
There exists a fixed, numerical constant $C'>0$ such that
\[
\min_{\substack{\mathbf{u}(\mathbf{x},t)\\
\left\langle |\mathbf{u}|^{2}\right\rangle =\emph{Pe}^{2}\\
\mathbf{u}=\mathbf{0}\text{ at }\partial D
}
}\,\left\langle |\nabla T|^{2}\right\rangle \leq\frac{C'}{\emph{Pe}}\cdot\left\langle |f|^{2}+|\nabla f|^{2}\right\rangle 
\]
whenever $\emph{Pe}\geq1$. The same bound holds using no-penetration conditions
$\mathbf{u}\cdot\hat{\mathbf{n}}=0$ in place of the no-slip ones
$\mathbf{u}=\mathbf{0}$ at $\partial D$.
\end{prop}
\begin{proof}
Pure convection roll-like flows occur as the simplest case of our
branching flows from \prettyref{subsec:Branching-flows}, with $n=1$
and using only $l_{1}=l_{\text{bulk}}$ and $\delta_{1}=\delta_{\text{bl}}$.
The rescaled velocities 
\[
\lambda_{\text{Pe}}\mathbf{u}\quad\text{with}\quad\lambda_{\text{Pe}}=\frac{\text{Pe}}{\sqrt{\left\langle |\mathbf{u}|^{2}\right\rangle }}
\]
generate temperature fields $T_{\text{Pe}}$ satisfying 
\[
\left\langle |\nabla T_{\text{Pe}}|^{2}\right\rangle \lesssim C_{0}(f)\cdot \left[l_{\text{bulk}}^{2}+\delta_{\text{bl}}+\frac{1}{\text{Pe}^{2}}\left(\frac{1}{l_{\text{bulk}}^{2}}+\frac{\delta_{\text{bl}}}{l_{\text{bulk}}^{2}}\right)\right]
\]
so long as $l_{\text{bulk}}\lesssim\delta_{\text{bl}}$, where now
 
\[
C_{0}=\left\langle |f|^{2}+|\nabla f|^{2}\right\rangle .
\]
The proof of this is contained in that of \prettyref{prop:crucial-estimate},
once one notes that the un-scaled velocities from \prettyref{subsec:Branching-flows}
obey
\[
\left\langle |\mathbf{u}|^{2}\right\rangle \lesssim1.
\]
To see this write 
\[
\mathbf{u}=\chi_{1}\mathbf{u}_{1}+\chi_{1}'\psi_{1}\hat{\mathbf{e}}_{\theta}
\]
in the case $n=1$, and use the definitions to get that
\[
|\mathbf{u}|\lesssim1+\frac{l_{\text{bulk}}}{\delta_{\text{bl}}}\lesssim1.
\]
(The proof for $n>1$ is the same, but we leave it to the reader as
it is not needed here.) The dependence of $C_{0}$ on $f$ comes from
\prettyref{lem:mean-flux-avg} and \prettyref{lem:quad-form} and
not, in this case, from \prettyref{lem:enstrophy-estimates}. Optimizing
gives 
\[
\delta_{\text{bl}}\sim l_{\text{bulk}}\sim\frac{1}{\text{Pe}^{1/2}}
\]
and this proves the result.
\end{proof}

\bibliographystyle{amsplain}
\bibliography{fluidsrefs}

\end{document}